\documentclass[11pt]{article}
\usepackage[ansinew]{inputenc}
\usepackage{graphicx}
\usepackage{color}

\usepackage{enumerate,latexsym}
\usepackage{latexsym}
\usepackage{amsmath,amssymb}% CUIDADO: Si uso estos packages, $\widetilde{}$ no funciona a veces
\usepackage{graphicx}
\usepackage{float}
\newfont{\bb}{msbm10 at 11pt}
\newfont{\bbsmall}{msbm8 at 8pt}
\def\rth{\mathbb{R}^3}
\def\R{\mathbb{R}}
\def\B{\mathbb{B}}
\def\N{\mathbb{N}}
\def\Z{\mathbb{Z}}
\def\C{\mathbb{C}}

\def\D{\mathbb{D}}

\def\esf{\mathbb{S}}

\newcommand{\T}{\mathbb{T}}

\newcommand{\ben}{\begin{enumerate}}
\newcommand{\bit}{\begin{itemize}}
\newcommand{\een}{\end{enumerate}}
\newcommand{\eit}{\end{itemize}}
\newcommand{\wh}{\widehat}

\newcommand{\Div}{\mbox{\rm div}}

\newcommand{\ed}{\end{document}}

\def\a{{\alpha}}

\def\t{{\theta}}
\def\g{{\gamma}}
\def\G{{\Gamma}}
\def\l{{\lambda}}

\def\ve{{\varepsilon}}
\def\cF{{\cal F}}

\def\centerbmp#1#2#3{\vskip#2\relax\centerline{\hbox to#1{\special
    {bmp:#3 x=#1, y=#2}\hfil}}}

\newtheorem{theorem}{Theorem}[section]
\newtheorem{lemma}[theorem]{Lemma}
\newtheorem{proposition}[theorem]{Proposition}

\newtheorem{remark}[theorem]{Remark}

\newenvironment{proof}{\smallskip\noindent{\it Proof.}\hskip \labelsep}
{\hfill\penalty10000\raisebox{-.09em}{$\Box$}\par\medskip}

\begin{document}
\begin{title}
{The Riemann minimal examples}
\end{title}
%\vskip .2in

\begin{author}
{William H. Meeks III,    \and Joaqu\'\i n P\' erez}
\end{author} \maketitle

\begin{abstract}
Near the end of his life, Bernhard Riemann made the marvelous
discovery of a 1-parameter family  $R_\l$, $\l\in (0,\infty)$,
of periodic properly embedded minimal surfaces in $\rth$ with
the property that every horizontal plane intersects
each of his examples in either a circle or a straight line. Furthermore, as the parameter $\l\to 0$
his surfaces converge to a vertical catenoid and as $\l\to \infty$ his surfaces converge to a vertical helicoid.
Since  Riemann's minimal examples are topologically planar domains
that are periodic with the fundamental domains for the associated $\Z$-action being diffeomorphic to a compact
annulus punctured in a single point, then topologically each of
these surfaces is diffeomorphic to the  unique genus zero
surface with two limit ends. Also he described  his surfaces analytically
in terms of elliptic functions on  rectangular elliptic curves.  This article exams
Riemann's original proof of the classification of  minimal surfaces foliated by circles and lines
in parallel planes
and  presents a complete outline of the recent proof that every
properly embedded minimal planar domain in $\rth$ is either a Riemann minimal example,
a catenoid, a helicoid or a plane.
\vspace{.3cm}


\end{abstract}

\section{Introduction.}
Shortly after the death of Bernhard Riemann, % around the age of 40,
 a large number of unpublished
handwritten papers were found in his office. Some of these papers were unfinished, but all of
them were of great interest  because of their profound insights and because of the deep and original
mathematical results
that they contained.
This  discovery of Riemann's handwritten unpublished manuscripts  led several of his students
and colleagues to rewrite these works, completing
any  missing arguments, and then to publish
them in their completed form in the Memoirs of the Royal Society of Sciences of G\"{o}ttingen as
a series of papers that began appearing in 1867.

One of these papers~\cite{ri2}, written by K. Hattendorf and M. H. Weber from Riemann's original notes
from the period 1860-61, was devoted to the theory of minimal surfaces in $\R ^3$. In one of these rewritten works,
Riemann %solved several Plateau problems, in other words, he
described  several examples of compact surfaces with boundary
that minimized their
area among all surfaces with the given explicit boundary. In the last section of this manuscript Riemann
tackled the problem of classifying those minimal surfaces which are bounded by a pair of circles in parallel
planes, under the additional hypothesis that every intermediate plane between the planes containing
the boundary circles also intersects the surface in a circle. Riemann proved that the only
solutions to this problem are (up to homotheties and rigid motions) the catenoid and a previously unknown
1-parameter family of surfaces $\{R_\lambda \ | \  \lambda \in \R \},$ known today as the
{\it Riemann minimal examples.} Later in 1869, A. Enneper~\cite{en1} demonstrated that there do not exist minimal surfaces
foliated by circles in a family of nonparallel planes, thereby completing the classification of the minimal
surfaces foliated by circles.

This purpose of this article is threefold. Firstly we will recover the original arguments by Riemann
by expressing them in more modern mathematical language (Section~\ref{riemann});
more specifically, we will provide  Riemann's analytic classification of
minimal surfaces  foliated by circles in parallel planes.
We refer the reader to Figure~\ref{R-image} for an image of a Riemann
minimal surface created from his mathematical formulas and produced by the graphics package in
{\tt Mathematica.}
Secondly we will
illustrate how the family of Riemann's minimal examples are still of great interest in the current state
of minimal surface theory in $\R^3$. Thirdly, we will indicate the key
results that have led to the recent proof that the plane, the helicoid,
the catenoid and the Riemann minimal examples are the only
 properly embedded minimal surfaces in $\R^3$ with the topology of
a planar domain; see Section~\ref{MPR} for a sketch of this proof.
In regards to this result, the reader should note that the plane and the helicoid are
conformally diffeomorphic to the complex
plane, and the catenoid is conformally diffeomorphic to the complex plane punctured in a single
point; in particular these three examples are surfaces of finite topology. However,
the Riemann minimal examples
are planar domains of infinite topology,
diffeomorphic to each other and characterized topologically as being diffeomorphic
to the  unique surface of genus zero with two limit ends.
The proof that the properly embedded minimal surfaces of
infinite topology and genus zero are Riemann minimal examples is due to Meeks, P\'erez and Ros~\cite{mpr6},
and it uses sophisticated ingredients from many branches of mathematics. Two essential ingredients
in their proof of this classification result are  Colding-Minicozzi theory, which concerns the geometry of
embedded minimal surfaces of finite genus, and the theory of integrable systems associated to the
Korteweg-de Vries equation.

In 1956 M. Shiffman~\cite{sh1} generalized in some aspects Riemann's classification theorem;
Shiffman's main result shows that a compact minimal annulus that
is bounded by two circles in parallel planes must be foliated by circles in the intermediate planes.
Riemann's result is more concerned with classifying analytically such minimal surfaces and his proof
that we give in Section~\ref{riemann} is simple and self-contained;
in Sections~\ref{subsec:shiffman} and~\ref{MPR} we will explore some
aspects of the arguments by Shiffman. After the preliminaries of
Section~\ref{preli},
we will include in Section~\ref{enneper}
the aforementioned reduction by Enneper from the general case
in which the surface is foliated  by circles to the case where the
circles lie  in parallel planes.
In Section~\ref{secgraphics} %{subsec:parametrizando} and \ref{subsec:representando}
we will introduce a
more modern viewpoint to study the Riemann minimal examples, which moreover will allow us to produce
graphics of these surfaces using the software package \textsf{Mathematica}. The
analytic tool for obtaining this graphical representation
will be the Weierstrass representation of a minimal surface, which is briefly described in
Section~\ref{subsec:weierstrass}. In general, minimal surfaces are geometrical objects that adapt  well
to rendering software, partly because from their analytical properties,
the Schwarz reflection principle for harmonic functions can be applied.
This reflection principle, that also explained in the preliminaries section,
allows one to represent relatively simple pieces of a minimal surface (where the computer graphics can achieve
high resolution), and then to generate the remainder of the  surface by simple reflections or
$180^\circ$-rotations
around lines contained in the boundaries of the simple fundamental piece.
On the other hand, as rendering software often represents a surface in space by producing the image
under the immersion of a mesh of points in the parameter domain, it especially important that we use
parameterizations whose associated meshes have natural properties,
such as preserving angles through the use of isothermal
coordinates.

\begin{remark}
{\rm
There is an interesting dichotomy between the situation in $\R^3$ and the one in $\R^n$, $n\geq 4$.
Regarding the natural $n$-dimensional generalization of the
problem tacked by Riemann, of producing a family of minimal hypersurfaces foliated by $(n-2)$-dimensional spheres,
W. C. Jagy~\cite{ja2} proved that if $n \geq 4$, then a minimal hypersurface in $\R^n$ foliated by hyperspheres
in parallel planes must be rotationally symmetric. Along this line of thought, we could say that the
Riemann minimal examples do not have a counterpart in higher dimensions.
}
\end{remark}

\noindent {\bf Acknowledgements}
The authors would like to express their gratitude to Francisco Mart\'\i n for
giving his permission to incorporate much of the material from a previous
joint paper~\cite{mape1} by him and the second author into the present manuscript. In particular, the
parts of our manuscript concerning
the classical arguments of Riemann and Weierstrass
 are largely rewritten translations of the paper~\cite{mape1}.

First author's financial support: This material is based upon work for
the NSF under Award No. DMS-1309236. Any opinions, findings, and
conclusions or recommendations expressed in this publication are those
of the authors and do not necessarily reflect the views of the NSF.
Second  author's financial support: Research partially supported by a MEC/FEDER
grant no. MTM2011-22547, and
Regional J. Andaluc\'\i a grant no. P06-FQM-01642.

\section{Preliminaries.} \label{preli}

Among the several equivalent defining formulations of  a minimal
surface $M\subset \R ^3$, i.e., a surface with
mean curvature identically zero, we  highlight the Euler-Lagrange equation
\[
(1+f_y^2)f_{xx}-2f_xf_yf_{xy}+(1+f_x^2)f_{yy}=0,
\]
where $M$ is expressed locally as the graph of a function $z=f(x,y)$ (in this paper we will use
the abbreviated notation
$f_x=\frac{\partial f}{\partial x},\partial f_{xx}$, etc., to refer to the partial derivatives
of any expression
with respect to one of its variables), and the formulation in local coordinates
\[
eG-2fF+Eg=0,
\]
where $\left( \begin{array}{cc}E & F \\ F & G \end{array}\right) , \left(
\begin{array}{cc}e & f \\ f & g \end{array}\right) $ are respectively the matrices of
the first and second fundamental forms of $M$ in a local parameterization. Of course, there are
other ways of characterize minimality such as by the harmonicity of the coordinate functions
or the holomorphicity of the Gauss map, but at this point it is worthwhile to remember  the
historical context in which the ideas that we wish to explain appeared.  Riemann was one of the founders of
the study of functions of one complex variable, and few things were well understood  in Riemann's time
concerning the relationship between minimal surfaces and holomorphic or harmonic functions.
Instead, Riemann imposed minimality by expressing the surface in implicit form, i.e., by espressing it
as the zero set of a smooth function $F\colon O\to \R $ defined in an open set $O\subset \R^3$, namely
\begin{equation}
\label{eq:div}
\Div \left( \frac{\nabla F}{| \nabla F| }\right) =0 \quad \mbox{ in $O$,}
\end{equation}
where  $\Div $ and $\nabla $ denote divergence and gradient in $\R ^3$, respectively.
The derivation of equation (\ref{eq:div}) is standard, but we next derive it for the
sake of completeness. First note that
\[
\Div \left( \frac{\nabla F}{| \nabla F| }\right) =\frac{1}{| \nabla F| }
\left( \Delta F-\frac{1}{| \nabla F| }(\nabla F)(| \nabla F
| ) \right) ,
\]
where $\Delta $ is the laplacian on $M$; hence (\ref{eq:div}) follows directly from the next lemma.
\begin{lemma}
\label{equivalencia}
If $0$ is a regular value of a smooth function $F\colon O \rightarrow \R$, then the surface
$M=F^{-1}(\{0 \})$ is minimal if and only if $| \nabla F| \Delta F=(\nabla F)(| \nabla F| )$ on $M$.
\end{lemma}
\begin{proof}
Since the tangent plane to $M$ at a point $p\in M$ is $T_pM=\ker (dF_p)=\langle (\nabla F)_p\rangle ^{\perp },$
then $(\nabla F)|_M$ is a nowhere zero vector field normal to $M$, and
$N=( \frac{\nabla F}{|\nabla F|}) |_M$ is a Gauss map for $M$. On the other hand,
\begin{equation}
\label{eq:lema1-1}
\Delta F=\mbox{trace}(\nabla ^2F)=\sum _{i=1}^2(\nabla ^2F)(E_i,E_i)+(\nabla ^2F)(N,N)
\end{equation}
where $\nabla ^2F $ is the hessian of $F$ and $E_1,E_2$ is a local orthonormal frame tangent
to $M$. As $(\nabla ^2F)(E_i,E_i)=\langle E_i(\nabla F)
,E_i\rangle =\langle E_i(|\nabla F| N),E_i\rangle =| \nabla F| \langle dN(E_i),E_i\rangle $,
then
\begin{equation}
\label{eq:lema1-2}
\sum _{i=1}^2(\nabla ^2F)(E_i,E_i)=| \nabla F|\mbox{ trace}(dN)=-2| \nabla F| H,
\end{equation}
with $H$ the mean curvature of $M$ with respect to $N$. Also,
$|\nabla F|^2(\nabla ^2F)(N,N)=(\nabla ^2F)(\nabla F,\nabla F)=
\frac{1}{2}(\nabla F)(| \nabla F| ^2)=
| \nabla F| (\nabla F) (| \nabla F| ) $. Plugging this formula
and (\ref{eq:lema1-2}) into (\ref{eq:lema1-1}), we get
$| \nabla F| \Delta F=-2| \nabla F| ^2H+(\nabla F)(| \nabla F| )$, and the proof
is complete.
\end{proof}

\subsection{Weierstrass representation.}
\label{subsec:weierstrass}
In the period 1860-70, Enneper and Weierstrass obtained representation formulas
for minimal surfaces in $\R^3$
by using curvature lines as parameter lines. Their formulae have become fundamental in the study of
orientable minimal surfaces (for nonorientable minimal surfaces there are similar formulations,
although we will not describe them  here). The reader can find a detailed explanation of the
Weierstrass representation in treatises on minimal surfaces by Hildebrandt et al.~\cite{dhkw1},
Nitsche~\cite{ni2} and Osserman~\cite{os1}. The starting point is the well-known formula
\[
\Delta X =2 \, H \, N,
\]
valid for an isometric immersion $X=(x_1,x_2,x_3)\colon M \to \R ^3$ of a Riemannian surface
into Euclidean space, where $N$ is the Gauss map, $H$ is the mean curvature and $\Delta X=
(\Delta x_1,\Delta x_2,\Delta x_3)$ is the Laplacian of the immersion. In particular,
minimality of $M$ is equivalent to the harmonicity of the coordinate functions $x_j$, $j=1,2,3$.
In the sequel, it is worth considering $M$ as a Riemann surface. We will also denote by $i=\sqrt{-1}$
and Re, Im will stand for real and imaginary parts.

We denote by $x_j^*$ the harmonic conjugate of $x_j$, which is locally well-defined up to additive constants.
Thus,
\[
\phi_j :=d x_j +i\,d x_j^\ast  ,
\]
is a holomorphic 1-form, globally defined on $M$. If we choose a base point $p_0\in M$, then the
equality
\begin{equation}
\label{eq:WR}
X(p)=X(p_0)+\mbox{Re} \int_{p_0}^p (\phi_1,\phi_2,\phi_3 ), \quad p \in M,
\end{equation}
recovers the initial minimal immersion,  where integration  in (\ref{eq:WR})
does not depend on the path in $M$ joining $p_0$ to $p$.

The information encoded by $\phi _j$, $j=1,2,3$, can be expressed with only two
pieces of data (this follows
from the relation that $\sum _{j=1}^3\phi _j^2=0$); for instance, the meromorphic function
$g=\frac{\phi_3}{\phi_1-i\phi_2} $ together with $\phi _3$ produce the other two 1-forms by means of
the formulas
\begin{equation}
\label{eq:phij}
\phi_1  =  \frac{1}{2} \left(\frac 1g-g \right) \phi_3, \quad
\phi_2  = \frac{i}{2} \left(\frac 1g+g \right) \phi_3,
\end{equation}
with the added bonus that $g$ is the stereographic projection of the Gauss map $N$ from the north
pole, i.e.,
\[
N=\left( 2 \: \frac{\mbox{Re}(g)}{1+|g|^2},  2 \:  \frac{\mbox{Im}(g)}{1+|g|^2},
  \frac{|g|^2-1}{1+|g|^2} \right) . \]

We finish this preliminaries section with the statement of the celebrated Schwarz reflection principle
that will be useful later. In 1894, H. A. Schwarz adapted his reflection principle for real-valued
harmonic functions in open sets of the plane to obtain the following result
for minimal surfaces. The classical proof of this principle can be found in
Lemma 7.3 of~\cite{os1}, and a different proof, based on the so-called Bj\"{o}rling problem,
can be found in \S$3.4$ of~\cite{dhkw1}.

\begin{lemma}[Schwarz]
\label{th:schwarz}
Any straight line segment (resp. planar geodesic) in a minimal surface is an axis of  a
$180^\circ$-rotational symmetry
(resp. a reflective symmetry in the plane containing the geodesic) of the surface.
\end{lemma}

\section{Enneper's reduction of the classification problem to foliations by circles in parallel planes.}
\label{enneper}
We will say that a surface $M \subset \R ^3$ is foliated by circles if it can be parameterized as
\begin{equation}
\label{para}
X(u,v)= c(u)+r(u) \left( \cos v  \, {\mathbf v}_1(u)+\sin v  \, {\mathbf v}_2(u)
\right), \quad u_0 <u<u_1, \; 0 \leq v<2 \pi,
\end{equation}
where $c\colon (u_0,u_1)\to \R^3$ is a curve that parameterizes the centers of the circles in the
foliation, and $r\colon (u_0,u_1)\to (0,\infty )$ is the radius of the foliating circle. Let
${\mathbf v}_1(u)$ and ${\mathbf v}_2(u)$ be  an orthonormal basis of the linear
subspace associated to the affine plane that contains the foliating circle. The purpose of this
section is to prove Enneper's reduction.

\begin{proposition}[Enneper]
\label{th:enneper}
If a minimal surface $M \subset \R ^3$ is foliated by circles, then these circles are
contained in parallel planes.
\end{proposition}
\begin{proof}
Let $\{\mathcal{C}_u \ | \ u\in (u_0,u_1) \}$ be the smooth, 1-parameter family of circles
that foliate $M$. Let $t\colon (u_0,u_1)\to \esf^2(1)$ be the unit normal vector to the plane that contains
$\mathcal{C}_u$. It suffices to show that $t(u)$ is constant. Arguing by contradiction, assume that
$t$ is not constant. After possibly restricting to a subinterval, we can suppose that $t'(u) \neq 0$
for all $u \in (u_0,u_1)$. Take a curve $\gamma:(u_0,u_1) \to \R ^3$ with $\gamma'(u)=t(u)$, for all
$u \in (u_0,u_1)$. The condition that $t'(u)$ vanishes nowhere implies that the curvature function
$\kappa (u)$ of $\g $ is everywhere positive.
Let $n(u)$, $b(u)$ be the normal and binormal (unit) vectors to $\g$, i.e.,
$\{ t,n,b\} $ is the Frenet dihedron of $\g $, and let $\tau $ be the torsion of $\g $.
The surface $M$ can be written as in (\ref{para})
with ${\mathbf v}_1=n$ and ${\mathbf v}_2=b$. Our purpose is to express the minimality condition
$H=0$ in terms of this parameterization. To do this, denote by
\[
\begin{array}{ccc}
E= |X_u|^2, & F=\langle X_u,X_v \rangle, & G=|X_v|^2,
\\
\rule{0cm}{.6cm}
e=\frac{\det (X_u,X_v,X_{uu})}{|X_u \wedge X_v|}, &
f=\frac{\det(X_u,X_v,X_{uv})}{|X_u \wedge X_v|} , &
g=\frac{\det (X_u,X_v,X_{vv})}{|X_u\wedge X_v|}
\end{array}
\]
the coefficients of the first and second fundamental forms of $M$ in the parameterization~$X$.

If  $(\alpha,\beta,\delta)$ corresponds to the coordinates of the velocity vector $c'$
of the curve of centers $c(u)$ with respect to the orthonormal basis $\{t,n,b \}$, then
a straightforward computation that only uses the Frenet equations for $\g $ leads to
\begin{multline*}
e\,G-2 \, f \, F+E \, g= \tfrac{1}{{|X_u \wedge
X_v|}}\left({a_1}\,\cos (3\,v)+ {a_2}\,\sin (3\,v)+   \right. \\ \left.
{a_3}\,\cos (2\,v) + {a_4}\,\sin (2\,v) +{a_5}\,\cos v+{a_6}\,\sin
v+{a_7}\right),
\end{multline*}
where the functions $a_j$, $j=1, \ldots,7$ depend only on the parameter $u$ and are given in terms
of the radius $r(u)$ of ${\cal C}_u$ and the curvature $\kappa (u)$ and torsion $\tau (u)$ of $\g$ by
\begin{eqnarray*}
a_1 & = & -\tfrac{1}{2} {r}^3\,\kappa \,\left( {\beta }^2 - {\delta }^2 +
{r}^2\,{\kappa }^2 \right),\\
a_2 & = & -{r}^3\,\beta \,\delta \,\kappa,\\
a_3 & = & \tfrac{{r}^3}{2}\,\left( -6\,\beta \,\kappa \,r' + r\,\left( 5\,\alpha
\,{\kappa }^2 + \kappa \,\beta ' - \beta \,\kappa ' \right)  \right), \\
a_4 & = & \tfrac{{r}^3}{2}\,\left( r\,\kappa \,\delta ' - \delta \,\left( 6\,\kappa
\,r' + r\,\kappa ' \right)  \right),\\
a_5 & = & -\tfrac{{r}^2}{2} \,\left( 3\,{r}^3\,{\kappa }^3 - 4\,\alpha \,\beta \,r'
+
\right. \\
& &     \left.   r\,\left( 8\,{\alpha }^2\,\kappa  + 3\,{\beta }^2\,\kappa  +
           3\,\kappa \,\left( {\delta }^2 + 2\,{(r')}^2 \right)  -   2\,\beta
\,\alpha ' + 2\,\alpha \,\left( \delta \,\tau  + \beta ' \right)  \right)  +
        2\,{r}^2\,\left( r'\,\kappa ' - \kappa \,r'' \right)  \right) ,
\\
a_6 & = & {r}^2\,\left( 2\,\alpha \,\delta \,r' + {r}^2\,\kappa \,\tau \,r' +
    r\,\left( \delta \,\alpha ' + \alpha \,\left( \beta \,\tau  - \delta '
\right)  \right)  \right),\\
a_7 & = & \tfrac{{r}^2}{2} \,\left( 2\,{\alpha }^3 + r\,\left( 2\,r'\,\left(
-2\,\beta \,\kappa  + \alpha ' \right)  + r\,\left( \kappa \,\left( 2\,\delta \,\tau
+ \beta ' \right)  - \beta \,\kappa
' \right)  \right)  + \right. \\
    & &   \left. \alpha \,\left( 2\,{\beta }^2 + 2\,{\delta }^2 + 5\,{r}^2\,{\kappa
}^2 + 2\,{(r')}^2 - 2\,r\,r'' \right)
          \right).
\end{eqnarray*}

As the functions in $\{ \cos(n \, v), \sin ((n+1) \, v)\ | \ n \in \N \cup \{ 0 \} \} $ are linearly
independent, then the condition $H=0$ is equivalent to $a_j= 0$, for each $j=1, \ldots,7$.
Since $r>0$ and $\kappa >0$, then the conditions $a_1=a_2 =0$ above imply
that $\beta=0$ and $\delta^2=r^2 \kappa^2$. Plugging this into $a_4=0$, we get $5 r \, r' \, \kappa^2=0$,
from where $r'=0$. Substituting this into
$a_3=0$ we get $5 r \, \alpha \, \kappa^2=0$, hence $\alpha=0$. Finally, plugging
$\alpha=\beta= r'=0$ and $\delta^2=r^2 \kappa^2$ into $a_5=0$, we deduce that $-3 r^5\kappa^3=0$, which
is a contradiction. This contradiction shows that $t(u)$ is constant and the proposition is proved.
\end{proof}

\section{The argument by Riemann.}
\label{riemann}
In this section we explain the classification by Riemann of the minimal surfaces in $\R^3$ that are
foliated by circles. By Proposition~\ref{th:enneper}, we can assume that the foliating circles of our minimal
surface $M\subset \R^3$ are contained in parallel planes, which after a rotation we will  assumed to be  horizontal.
We will take the height $z$ of the plane as a parameter of the foliation, and denote by $(\a (z),z)=
(\a _1(z),\a _2(z),z)$ the center of the circle $M\cap \{ x_3=z\} $ and $r(z)>0$ its radius;
here we are using the identification $\R ^3=\R^2\times \R $. The functions $\a _1,\a _2,r$ are assumed to be
of class $C^2$ in an interval $(z_0,z_1)\subset \R$.

Consider the function $F:\R^2\times (z_0,z_1)\to \R$
given by
\begin{equation}
\label{eq:F}
 F(x,z)=| x-\a (z)| ^2-r(z)^2,
\end{equation}
where $x=(x_1,x_2)\in \R^2$. So, $M\subset F^{-1}(\{ 0\} )$ and thus, Lemma~\ref{equivalencia}
 gives that $M$ is minimal if and only if
\[
F_z^2+| x-\a | ^2\left( 2+F_{zz}\right) +2\langle x-\a ,\a '\rangle F_z=0 \quad \mbox{ in $M$}.
\]
As $| x-\a | ^2=r^2$ in $M$ and $2\langle x-\a ,\a '\rangle
=-\left( | x-\a | ^2\right) _z=-(F+r^2)_z=-F_z-(r^2)'$, then the minimality of $M$ can be written as
\begin{equation}
\label{eq:teor2-3}
2 r^2+r^2F_{zz}-(r^2)'F_z=0.
\end{equation}

The argument amounts to integrating (\ref{eq:teor2-3}). First we divide by $r^4$,
\[
0=\frac{2}{r^2}+\frac{r^2F_{zz}-(r^2)'F_z}{r^4}=\frac{2}{r^2}+\left( \frac{F_z}{r^2}\right) _z,
\]
and then we integrate with respect to $z$, obtaining
\begin{equation}
\label{eq:teor2-4}
f(x)=2\int ^z\frac{du}{r(u)^2}+\frac{F_z(x,z)}{r(z)^2},
\end{equation}
for a certain function of $x$. On the other hand, (\ref{eq:F}) implies that
\[
\frac{\partial }{\partial x_j}\left(
\frac{F_z}{r^2}\right) = -\frac{2\a _j'}{r^2},
\]
which is a function depending only on $z$, for each $j=1,2$. Therefore,
for $z$ fixed, the function $x\mapsto \frac{F_z}{r^2}$ must be
affine. As $\int ^z\frac{du}{r(u)^2}$ only depends on $z$, we conclude from
(\ref{eq:teor2-4}) that $f(x)$ is also an affine function of $x$, i.e.,
\begin{equation}
\label{eq:teor2-5}
2\int ^z\frac{du}{r(u)^2}+\frac{F_z(x,z)}{r(z)^2}=2\langle a,x\rangle +c
\end{equation}
for certain $a=(a_1,a_2)\in \R^2$, $c\in \R $. Taking derivatives in (\ref{eq:teor2-5})
with respect to $x_j$, we get $\a _j'(z)=-a_jr^2(z)$; hence
\[
\a (z)=-m(z)\, a \quad \mbox{ where }m(z)=\int ^zr(u)^2\, du.
\]
These formulas determine the center of the circle $M\cap \{ x_3=z\} $ up to a horizontal translation
that is independent of the height.

In order to determine the radius of the circle $M\cap \{ x_3=z\} $ as a function of $z$, we come back to
equation (\ref{eq:teor2-3}). Since $F_z=-2\langle x-\a ,\a '\rangle -(r^2)'$, then
$F_{zz}=2| \a '| ^2-2\langle x-\a ,\a ''\rangle -(r^2)''$. Plugging $\a
'=-r^2\, a$ into these two expressions, and this one into (\ref{eq:teor2-3}), we deduce that
the minimality of $M$ can be rewritten as
\begin{equation}
\label{eq:teor2-6}
2| a|^2r^6+((r^2)')^2+r^2(2-(r^2)'')=0,
\end{equation}
which is an ordinary differential equation in the function $q(z)=r(z)^2$. Solving this ODE is straightforward:
first note that
\[
\left( \frac{(q')^2}{q^2}\right) '=\frac{2q'}{q^3}\left[ -(q')^2+qq''\right]
\stackrel{(\ref{eq:teor2-6})}{=}
\frac{2q'}{q^3}\left( 2| a| ^2q^3+2q\right) =4\left( | a| ^2q'+\frac{q'}{q^2}\right).
\]
Hence integrating with respect to $z$ we have
\begin{equation}
\label{eq:teor2-7}
\frac{(q')^2}{q^2}=4\left( | a| ^2q-\frac{1}{q}\right) +4\l ,
\end{equation}
for certain $\l \in \R $. In particular, the right-hand side of (\ref{eq:teor2-7})
must be nonnegative. Solving for $q'(z)$ we have
\[
q'=\frac{dq}{dz}=2\sqrt{| a| ^2q^3-q+\l q^2};
\]
hence the height differential of $M$ is
\[
dz=\frac{1}{2}\frac{dq}{\sqrt{| a| ^2q^3-q+\l q^2}},
\]
where $q=r^2$. Viewing $q$ as a real variable, the third coordinate function of $M$
can be expressed as
\begin{equation}
\label{eq:teor2-8}
z(q)=\frac{1}{2}\int ^q\frac{du}{\sqrt{| a| ^2u^3-u+\l u^2}}.
\end{equation}

To obtain the first two coordinate functions of $M$, recall that $M\cap \{ x_3=z\} $ is a circle
centered at $(\a ,z)$ with radius $r$. This means that besides the variable $q$ that gives the height
$z(q)$, we need another  real variable to parameterize the circle centered at $(\a ,z)$ with radius $r$:
\[
x(q,v)=\a (z(q))+\sqrt{q}e^{iv}=-m(z(q))\, a+\sqrt{q}e^{iv}=-\int ^{z(q)}q(z)\, dz\cdot a+\sqrt{q}e^{iv}
\]
\[
=-\int ^{q}q\frac{dz}{dq}\, dq\cdot a+\sqrt{q}e^{iv}=-\frac{1}{2}\int ^q\frac{u \, du}
{\sqrt{| a| ^2u^3-u+\l u^2}}\cdot a
+\sqrt{q}e^{iv},
\]
where $0\leq v<2\pi $. In summary, we have obtained the following parameterization  $X=(x_1,x_2,z)=(x_1+ix_2,z)$ of
 $M$:
\begin{equation}
\label{eq:teor2-9}
X(q,v)=f(q)(a,0)+\sqrt{q}(e^{iv},0)+(0,z(q)),
\end{equation}
where ${\displaystyle
f(q)=-\frac{1}{2}\int ^q\frac{u \, du}{\sqrt{| a| ^2u^3-u+\l u^2}}}$
and $z(q)$ is given by  (\ref{eq:teor2-8}).

The surfaces in (\ref{eq:teor2-9}) come in a 3-parameter family depending on the
values of $a_1,a_2,\l $. Nevertheless, two of these parameters correspond to rotations and homotheties
of the same geometric surface, and so there is only one genuine geometric parameter.

Next we will analyze further the surfaces in the family (\ref{eq:teor2-9}), in order
to understand their shape and other properties. First observe that the first
term in the last expression of $X(q,v)$ parameterizes the center of the circle
$M\cap \{ x_3=z(q)\} $ as if it were placed at height zero,
the second term parameterizes the circle itself ($q$ is positive as $q(z)=r(z)^2$),
and the third one places the circle at height $z(q)$.
To study the shape of the surface, we will analyze for which values of $q>0$
the radicand $| a| ^2q^3-q+\l q^2$ is nonnegative (this is a necessary condition, see (\ref{eq:teor2-7})).
This indicates that the range of $q$ is of the form $[q_1,+\infty )$ for certain $q_1>0$. Also, the
positivity of the integrand in (\ref{eq:teor2-8}) implies that $z(q)$ is increasing. Since choosing a
starting value of $q$ for the integral (\ref{eq:teor2-8}) amounts to vertically translating the surface,
we can choose this starting value as $q_1$, which geometrically means:
\begin{quote}
{\it We normalize $M$ so that the circle of minimum radius in $M$ is at height zero,
and $M$ is a subset of the
half-space $\{ (x_1,x_2,x_3) \mid x_3\geq 0\} $.}
\end{quote}
This translated surface will have a lower boundary component being a circle
(or in the limit case  a straight line) contained in the plane
$\{ z=0\} $ (in particular, $M$ is not complete).
If we choose the negative branch of the square root
when solving for $q'(z)$
after (\ref{eq:teor2-7}), we will obtain another surface $M'$ contained in the half-space
$\{ x_3\leq 0\} $ with the same boundary as $M$ in
$\{ x_3=0\} $. The union of $M$ with $M'$ is again a smooth minimal surface;
this is because  the tangent spaces to both $M,M'$ coincide along
the common boundary. Nevertheless, $M\cup M'$ might fail to be complete;
we will obtain more information about this issue of completeness
when discussing the value of $a$.

Analogously, picking a starting value of $q$ for the integral in
(\ref{eq:teor2-9}) that gives $f$,
corresponds to translating horizontally the center of the circle
$M\cap \{ x_3=z(q)\} $ by a vector independent of  $q$,
or equivalently, translating $M$ horizontally in $\R^3$.
Thus, we can normalize this starting value of $q$ for the integral
in (\ref{eq:teor2-9}) to be the same $q_1$ as before. This means that
we may assume:
\begin{quote}
{\it The circle of minimum radius in $M$ has its center at the origin of $\R ^3$.}
\end{quote}

We next discuss cases depending on whether or not $a$ vanishes.
\par
\vspace{.2cm}
{\sc Case I: $a=(0,0)$ gives the catenoid.}
\par
\vspace{.2cm}
\noindent
In this case, (\ref{eq:teor2-9}) reads
$X(q,v)=\sqrt{q}(e^{iv},0)+(0,z(q))$, where $z(q)=\frac{1}{2}\int_{q_1}^q\frac{du}{\sqrt{-u+\l u^2}}du$.
In particular, $M$ is a surface of revolution around the $z$-axis, so $M$ is a half-catenoid with neck
circle at height zero. In order to determine $q_1$, observe that $\l $ is positive as follows from (\ref{eq:teor2-7})),
and that the function $u\mapsto -u+\l u^2$ is nonnegative in $(-\infty ,0]\cup [\frac{1}{\l },\infty )$.
Therefore,
 we must take $q_1=\frac{1}{\l }$. Furthermore, the integral that defines $z(q)$ can be explicitly solved:
\[
z(q)=\frac{1}{2}\int _{1/\l }^q\frac{du}{\sqrt{-u+\l u^2}}du=\frac{1} {\l}\arg \sinh \,
\sqrt{-1+q\l },
\]
which gives the following parameterization of $M$:
\begin{equation}
\label{eq:catenoide1}
X(q,v)=\left( \sqrt{q}e^{iv},\frac{1}{\l }\arg \sinh\,
\sqrt{-1+q\l }\right) .
\end{equation}
In this case, the surface $M'\subset \{ x_3\leq 0\} $
defined by the negative branch of the square root when solving for $q'(z)$ in
(\ref{eq:teor2-7}) is the lower half of the same catenoid of which $M$ is the upper part.
In particular, $M\cup M'$ is complete.
\par
\vspace{.2cm}
{\sc Case II: $a\neq (0,0)$ gives the Riemann minimal examples.}
\par
\vspace{.2cm}
\noindent
As $f,z$ depend on $| a| ^2$ but not on $a=(a_1,a_2)$, then a rotation of
$a\equiv a_1+ia_2$ in $\R^2 \equiv \C $ around the origin
by angle $\t $ will leave $f$ and $z$ invariant. By  (\ref{eq:teor2-9}),
the center of the circle $M\cap \{ x_3=z(q)\} $ will be
also rotated by the same angle $\t$ around the $x_3$-axis,
while the second and third terms in (\ref{eq:teor2-9}) will remain the same.
This says that rotating $a$ in $\C $ corresponds to rotating $M$ in $\R ^3$ around the $x_3$-axis,
so without lost
of generality we can assume that $a\equiv (a,0)\in \R^2$ with $a>0$.

The radicand of equation (\ref{eq:teor2-8}) is now expressed by $q(a^2q^2-1+\l q)$,
hence $a^2q^2-1+\l q\geq 0$. This occurs for
$q \in \left[ \frac{1}{2a^2}(-\l -\sqrt{4a^2+\l ^2}),0\right] \cup \left[ q_1,\infty \right)$,
where
\[
q_1=\frac{1}{2a^2}(-\l +\sqrt{4a^2+\l ^2}).
\]
As $q>0$, then the correct range is $q\in [q_1,\infty )$.
Now fix the starting integration values for $z(q),f(q)$ at $q_1$,
and denote by $z_{a,\l },f_{a,\l }$  the functions given by (\ref{eq:teor2-8}), (\ref{eq:teor2-9}),
respectively.
A straightforward change of variables shows that
\[
\sqrt{a}z_{a,\l }(q)=z_{1,\l /a}( aq) ,\qquad
a^{3/2}f_{a,\l }(q)=f_{1,\l /a}(aq).
\]
Thus, the minimal immersions $X_{a,\l },X_{1,\l /a}$ are related by a homothety of ratio $\sqrt{a}$:
\[
\sqrt{a}X_{a,\l }(q,v)=X_{1,\l /a}(aq,v).
\]
Therefore, we can assume $a=1$, i.e., our surfaces only depend on the real parameter $\l $:
\begin{equation}
\label{eq:Riemann1}
X_{\l }(q,v)=f_{\l }(q)(1,0)+\sqrt{q}(e^{iv},0)+(0,z_{\l }(q)),
\end{equation}
where
\[
f_{\l }(q)=-\frac{1}{2}\int _{q_1}^q\frac{u \, du}{\sqrt{u^3-u+\l u^2}}, \quad z_{\l
}(q)=\frac{1}{2}\int _{q_1}^q \frac{u \, du}{\sqrt{u^3-u+\l u^2}} \quad \mbox{and}
\]
\[
q_1=q_1(\l )=\frac{1}{2}\left( -\l +\sqrt{4+\l ^2}\right) .
\]
Calling $M_{\l }=X_{\l}([q_1,\infty )\times [0,2\pi ))$, the center of the circle $M_{\l }\cap \{
x_3=z_{\l }(q)\} $ lies en the plane $\Pi \equiv \{ x_2=0\} $, which implies that $M_{\l }$
is symmetric with respect to $\Pi $.

The increasing function $q\mapsto z_{\l }(q)$, $q\in [q_1,\infty )$, is bounded because
$\displaystyle\int ^q\frac{du}{\sqrt{u^3}}$ has limit zero as $q\rightarrow \infty $.
This means that $M_{\l }$ lies in
a slab of the form $$S(0,\zeta )=\{ (x_1,x_2,z) \in \R ^3 \ | \ 1\leq z\leq \zeta \}, $$ where
$\zeta =\zeta (\l )=\lim _{q\rightarrow \infty }z_{\l }(q)>0$.

We next analyze $M_{\l }\cap \{ x_3=\zeta \} $. As $M_{\l }$  is symmetric with respect to $\Pi $,
given $q\in [q_1,\infty [$, the circle $M_{\l }\cap \{ x_3=z(q)\} $ intersects $\Pi $
at two antipodal points
$A_+(q),A_-(q)$ that are obtained by imposing the condition $\sin v=0$ in (\ref{eq:Riemann1}), i.e.,
\[
A_{\pm }(q)=\left( -\frac{1}{2}\int _{q_1}^q\frac{u \, du}{\sqrt{u^3-u+\l u^2}}
\pm \sqrt{q},z_{\l }(q)\right) \in (\R \times \{ 0
\} )\times \R \subset \C \times \R .
\]
Since $q$ is not bounded, the first coordinate of  $A_-(q)$ tends to $-\infty $
as $q\to \infty $, that is to say, when we approach the upper boundary
plane $\{ x_3=\zeta \} $ of the slab $S(0,\zeta )$. On the contrary, the first
coordinate of  $A_+(q)$ tends to a finite limit as
$q\to \infty $, because for sufficiently large values of $q$ we have
\[
\left(
-\frac{1}{2}\int _{q_1}^q\frac{u \, du}{\sqrt{u^3-u+\l u^2}}+\sqrt{q}\right) \approx \left(
\mbox{constant}(\l )-\frac{1}{2}
\int _{q_1}^q\frac{du}{\sqrt{u}}+\sqrt{q}\right) =\mbox{constant}(\l ).
\]
Therefore, $A_+(q)$ converges as $q\to \infty $ to a point $A\in \{
x_3=\zeta \} \cap \Pi $. This proves that $M_{\l }\cap \{ x_3=\zeta \} \neq \mbox{\O }$.
As $M_{\l }\cap \{ x_3=\zeta \} $ cannot be compact (because $A_-(q)$ diverges in $\R ^3$ as $q\to \infty $),
we deduce that $M_{\l }\cap \{ x_3=\zeta \} $ is a noncompact limit
of circles symmetric with respect to $\Pi $, hence it is a straight
line $r$ orthogonal to $\Pi $ and  passing through $A$.

As the boundary of $M_{\l }$ consists of a circle in $\{ x_3=0\} $ and a  straight line
$r$ in $\{ x_3=\zeta \} $, then the Schwarz reflection
principle (Lemma~\ref{th:schwarz}) implies that $M\cup \mbox{Rot}_r(M)$ is a minimal surface,
where Rot$_r$ denotes the $180^\circ$-rotation
around $r$. Clearly, $M\cup \mbox{Rot}_r(M)\subset S(0,2 \, \zeta )$
has two horizontal boundary circles of the same radius.
The same behavior holds if we choose the negative branch of the square root
when solving (\ref{eq:teor2-7}) for $q'(z)$,
but now for a slab of the type $S(-\zeta ,0)$. This means that we can
rotate the surface by $180^\circ$ around the straight line
$r'= \{ -p\ | \ p\in r\} \subset \{ x_3=-\zeta \} $, obtaining a minimal surface
that lies in $S(-4\, \zeta ,2\, \zeta )$, whose boundary
consists of two circles of the same radius in the boundary planes of this slab
and such that the surface contains in its interior
three parallel straight lines at heights $\zeta$, $-\zeta $, $-3 \, \zeta$,
all orthogonal to $\Pi $.
Repeating this rotation-extension process
indefinitely we produce a complete embedded minimal surface $R_{\l }\subset \R ^3$
that contains parallel  lines contained in
the planes $\{ x_3=(2k+1)\zeta (\l )\} $ with $k\in \Z $ and orthogonal to $\Pi $.
It is also clear that $R_{\l }$ is invariant under the
translation $T_{\l}$ by the vector $2(A-B)$ (here $B=r'\cap \Pi $),
obtained by the composition of two $180^\circ$-rotations in consecutive lines.
This surface $R_{\l }$ is what we call a {\it Riemann minimal example,}
and it is foliated by circles and lines in parallel planes.

For each Riemann minimal example $R_{\l }$, the circles of minimum radius lie in the planes
$\{ x_3=2 \, k \, \zeta \,(\l )\} $, $k\in \Z $, and this minimum radius is
\[
\sqrt{q_1(\l )}=\sqrt{(-\l +\sqrt{4+\l ^2})/2}.
\]
This function of $\l \in \R$ is one-to-one (with negative derivative),
from where we conclude that
\begin{quote}
{\it
The Riemann minimal examples $\{ R_\l \ | \  \l \in \R \}$ form a 1-parameter
family of noncongruent surfaces.
}
\end{quote}
In Figure~\ref{fig:AB} we have represented the intersection of $R_{\l }$ for $\l =1$ with the
symmetry plane $\Pi $. At each of the points $A$ and $B$
there passes a straight line contained in the surface and orthogonal to $\Pi$.
\begin{figure}
\begin{center}
\includegraphics[height=5cm]{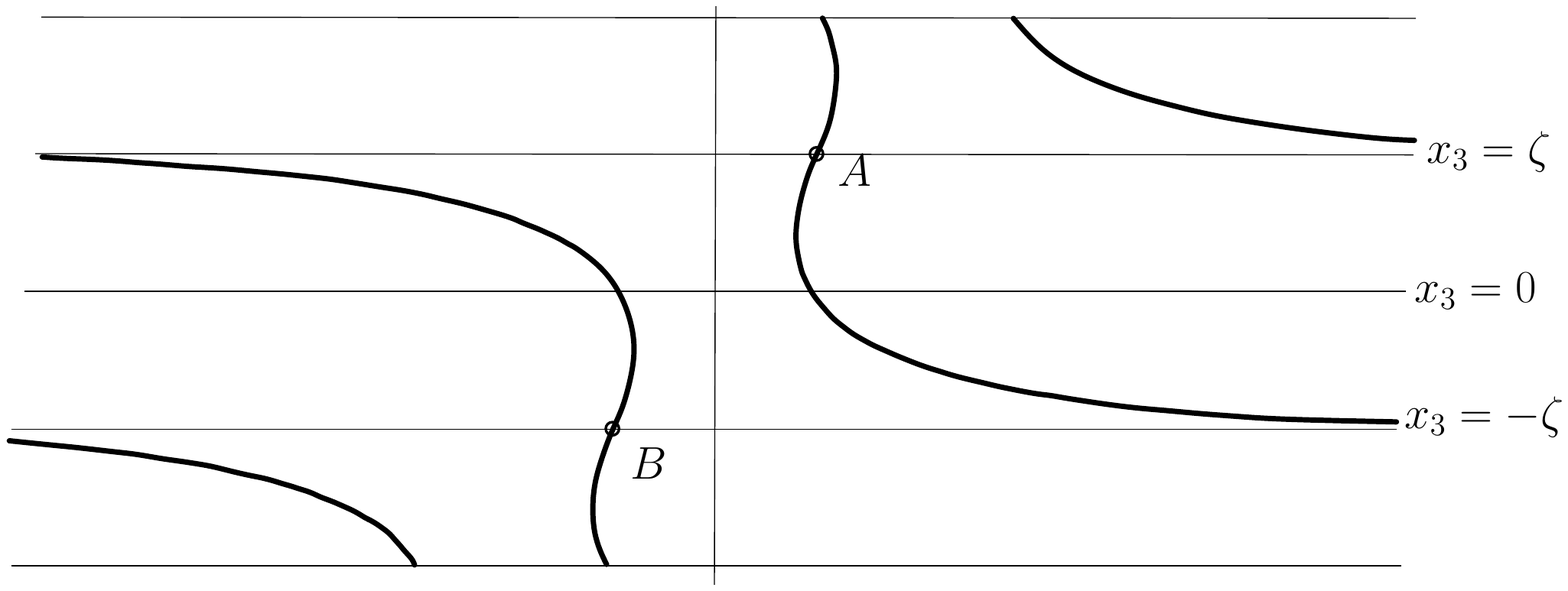}
\caption{The intersection of $R_\l$ with the symmetry plane $\Pi =\{x_2=0 \}$.
The translation of vector $2 (A-B)$ leaves $R_\l$ invariant.}
\label{fig:AB}
\end{center}
\end{figure}

Viewed as a complete surface in $\rth$, each Riemann minimal example $R_\lambda$ has the topology of an
infinite  cylinder punctured in an infinite discrete set of points, corresponding to its
planar ends, and $R_\lambda$ is invariant under a translation $T_\lambda$.
Furthermore, quotient surface
${\cal R}_\lambda=R_\lambda /T_\lambda$  in the
3-dimensional flat Riemannian manifold $\rth/T_\lambda$
is conformally diffeomorphic to a flat torus $\T$ punctured in two points.
In addition, the Gauss map $G$ of
${\cal R}_\lambda$ has degree two and has exactly two branch points. This means that
 $G$ is a meromorphic function on the punctured torus that extends to
a meromorphic function $\wh{G}$ of degree two on $\T$ whose zeros and poles of order two
occur at the two points corresponding to the planar ends of ${\cal R}_\lambda$ in $\rth/T_\lambda$.
It then follows from Riemann surface theory that the degree-two meromorphic
function $\wh{G}$ is a complex multiple of the
Weierstrass $\mathcal{P}$-function on the underlying elliptic curve $\T$.
Furthermore, the vertical  plane of symmetry and the rotational symmetry
around either of the two lines on the quotient surface ${\cal R}_\lambda$ imply that
$\T$ is conformally $\C/\Lambda$ where $\Lambda$ is a rectangular lattice.

\subsection{Graphics representation of the Riemann minimal examples.}
\label{secgraphics}

With the parameterizations of the Riemann minimal examples obtained in the
preceding section we can represent these surfaces with the help of the
 \textsf{Mathematica} graphing package. Nevertheless, we will not use the
parameterization in (\ref{eq:Riemann1}), mainly because the parameter $q$ diverges
when producing the straight lines contained in $R_{\l }$; instead, we will
use a conformal parameterization given by the Weierstrass representation.

Recall that ${\cal R}_\lambda=R_\lambda /T_\lambda$ is a twice punctured torus, and that its
Gauss map (which can be regarded as a holomorphic function on ${\cal R}_{\l }$, see
Section~\ref{subsec:weierstrass}) has degree two on the compactification.
In particular, the compactification of ${\cal R}_{\l }$ is conformally equivalent to
the following algebraic curve:
\[
\overline{M}_\sigma=\left\{ (z,w) \in \overline{\C} \times \overline{\C} \ | \
w^2 = z(z-1) (z+\sigma) \right\},
\]
where $\sigma \in \C-\{0,1\}$ depends on $\l $ in some way to be determined later, and
we can moreover choose the degree-two function $z$ on $\overline{M}_\sigma$ so that
meromorphic extension of the Gauss map of ${\cal R}_\lambda $ to $\overline{M}_\sigma$ is
$g^{\sigma}(z,w)= \rho \, z$, for a certain constant $\rho=\rho(\sigma) \in \C^*$. It is worth
mentioning that the way one endows $\overline{M}_\sigma$ with a complex structure is also due
to Riemann, when studying multivalued functions on the complex plane (in this case,
$z \mapsto \sqrt{z (z-1) (z+\sigma)}$).

Since the third coordinate function of $R_\lambda $ is harmonic and extends
analytically through the planar ends
(because it is bounded in a neighborhood of each end), its complex differential
$dx_3+i\, dx_3^\ast =\phi _3^{\sigma}$
is a holomorphic 1-form on the torus $\overline{M}_\sigma$ (without zeros).
As the linear space of
holomorphic 1-forms on a torus has complex dimension~1, then we deduce that
$\phi_3^{\sigma} \equiv \mu \frac{dz}{w}$
 for some $\mu  \in \C^*$ also to be determined. Clearly, after possibly
 applying a homothety to $R_{\l }$, we can assume that
$|\mu |=1$.

\subsubsection{Symmetries of the surface.}
\label{sec4.1.1}
Recall that each surface $R_\lambda$ is invariant under certain
rigid motions of $\R^3$, which therefore induce intrinsic isometries of
the surface. These intrinsic isometries are in particular conformal
diffeomorphisms of the algebraic curve $\overline{M}_{\sigma }$, that
might be holomorphic or anti-holomorphic depending on whether or not
they preserve or invert the orientation of the surface.
These symmetries will be useful in determining the constants that appeared in the above two paragraphs.

First consider the orientation-preserving isometry of $R_{\l }$ given by $180^\circ$-rotation
about a straight line parallel to the $x_2$-axis,
  that intersects $R_{\l }$ orthogonally at two points lying in a horizontal circle of minimum
  radius (these points would be represented
  by the mid point of the segment $\overline{AB}$ in Figure~\ref{fig:AB}).
  This symmetry induces an order-two biholomorphism
$S_1$ of $\overline{M}_\sigma$, that acts on $g^{\sigma},\phi_3^{\sigma}$ in the following way:
\[
g^{\sigma} \circ S_1=-1/g^{\sigma},\qquad
S_1^*\phi_3^{\sigma}=-\phi_3^{\sigma}.
\]
As $S_1$ interchanges the branch values of $g$, we deduce that $\rho= \frac{1}{\sqrt{\sigma}}$
and $S_1(z,w)=(-\frac{\sigma}{z},-\sigma \frac{w}{z^2})$.

Another isometry of $R_\l$ is the $180^\circ$-rotation about a straight line $r$ parallel to the
$x_2$-axis and contained in the surface.
As this symmetry reverses orientation of $R_\l$, then it induces an order-two anti-holomorphic diffeomorphism
$S_2$ of $\overline{M}_\sigma$, that acts on $g^{\sigma},\phi_3^{\sigma}$ as
\[
g^{\sigma} \circ S_2= \overline{g^{\sigma}}\qquad S_2^*\phi_3^{\sigma} = -\overline{\phi_3^{\sigma}}.
\]
$S_2$ fixes the branch points of $g^\sigma$ (one of them lies in $r$),
from where we get $\sigma \in \R$, $\mu  \in \{\pm 1, \pm i \}$ and
$S_2(z,w)=(\overline{z},\pm \overline{w})$. Furthermore, the unit normal vector to $R_{\l }$
along $r$ takes values in
$\esf^1(1)\cap \{x_2=0 \}$, which implies that $g^\sigma(1,0)\in \R$, and so, $\sigma>0$.

The following argument shows that we can assume that $\mu =1$. Consider the Weierstrass data $\left(\overline{M}_\sigma, \; g^{\sigma}, \; \phi_3^{\sigma}\right)$,
and the biholomorphism $\Sigma \colon \overline{M}_{1/\sigma} \to \overline{M}_{\sigma}$ given by
$\Sigma(z,w)=(-\sigma \, z, i \, \sigma^{3/2}\, w)$. Then, we have that
\[
g^{\sigma} \circ \Sigma= -g^{1/\sigma}, \quad \Sigma^* \phi_3^{\sigma}=
\frac{i}{\sqrt{\sigma}} \phi_3^{1/\sigma}.
\]
The change of variable via $\Sigma $ gives
\[
\int_{\Sigma(P_0)}^{\Sigma(P)} \frac{i}{\sqrt{\sigma}} \phi_3^{1/\sigma}=\int_{P_0}^{P}  \phi_3^{\sigma},
\]
and similar equations hold for the other two components $\phi_1^{1/\sigma}, \phi_2^{1/\sigma}$ of
the Weierstrass form (see~(\ref{eq:phij})).
This implies that we can assume after a rigid motion and homothety, that $\mu =1$ and that the isometry $S_2$ is
\[
S_2(z,w)=(\overline{z},-\overline{w}).
\]
Finally, the reflective symmetry of $R_{\l }$ with respect to the plane $\{x_2=0 \}$ induces an anti-holomorphic involution
$S_3$ of $\overline{M}_{\sigma }$ which fixes the branch points of $g^{\sigma }$ (including the zeros
and poles $(0,0),(\infty ,\infty )$ of $g$, which
correspond to the ends of ${\cal R}_{\l }$), and that preserves the third coordinate function. It is
then clear that this transformation has the form $S_3(z,w)=(\overline{z},\overline{w})$.

\subsubsection{The period problem.}
\label{sec4.1.2}
We next check that the Weierstrass data $\left(\overline{M}_\sigma, \; g^{\sigma}, \; \phi_3^{\sigma}\right)$
produces a minimal surface in $\R^3$, invariant under a translation vector.
A homology basis of the algebraic curve
$\overline{M}_\sigma$ is ${\cal B}=\{\gamma_1,\gamma_2\}$, where these closed curves are the liftings to $\overline{M}_\sigma$
through the $z$-projection of the curves in the complex plane represented in Figure~\ref{fig:curvas}.
\begin{figure}
\begin{center}
\includegraphics[width=0.5\textwidth]{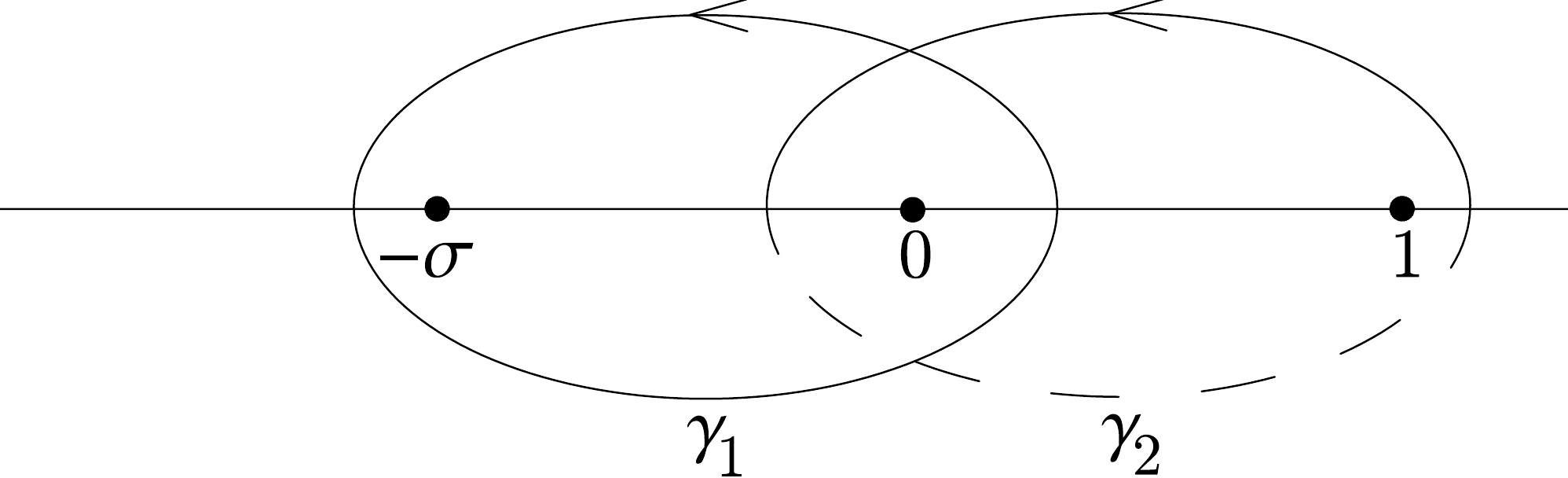}
\caption{The curves $\gamma_1$ and $\gamma_2$ in the $z$-complex plane.}
\label{fig:curvas}
\end{center}
\end{figure}

The action of the symmetries $S_j$, $j=1,2,3$, on the basis $\cal B$ and on the Weierstrass data implies that
\[
\mbox{Re}  \int_{\gamma_1} (\phi_1^{\sigma},\phi_2^{\sigma},\phi_3^{\sigma}) =0, \qquad \mbox{Re} \int_{\gamma_2} (\phi_1^{\sigma},\phi_2^{\sigma},\phi_3^{\sigma}) \in \{x_2=0 \}.
\]
In particular, the Weierstrass data $\left(\overline{M}_\sigma, \; g^{\sigma}, \; \phi_3^{\sigma}\right)$
gives rise to a well-defined
minimal immersion on the cyclic covering of $\overline{M}_\sigma -\{ (0,0),(\infty ,\infty )\} $ associated
to the subgroup of the fundamental group
 of $\overline{M}_\sigma$ generated by $\gamma_2$. We will call $M_{\sigma }\subset \R ^3$ to the image of
 this immersion; we will
 prove in Section~\ref{subsec:shiffman} that $M_{\sigma }$ is one of the Riemann minimal examples
 $R_{\l }$ obtained in Section~\ref{riemann}.

For the moment, we will content ourselves with finding a simply connected domain of $M_{\sigma }$
bordered by symmetry lines (planar geodesics or straight lines).
The reason for this is that the package \textsf{Mathematica} works with parameterizations defined
in domains of the plane, which once represented graphically,
can be replicated in space by means of the Schwarz reflection principles; we will call such a
simply connected domain of $M_{\sigma }$ a {\it fundamental piece.}
Having in mind how the isometries $S_2$ and $S_3$ act on the $z$-complex plane, it is clear that
we can reduce ourselves to representing graphically
the domain of $M_{\sigma }$ that corresponds to moving $z$ in the upper half-plane. Using the symmetry
$S_1$ we can reduce even further this domain, to
the set $ \left\{ z \in \C \ | \ \left| z-\frac{1-\sigma}{2} \right| \leq \frac{1+\sigma}{2}, \;
\mbox{Im} (z) \geq 0 \right\} $. As the point $(z,w)=(0,0)$
corresponds to an end of the minimal surface $M_{\sigma }$,
we will remove a small neighborhood of in the $z$-plane centered at the origin. In this way we
get a planar domain $\Omega_\sigma$ as in Figure~\ref{fig:omega}.
\begin{figure}
\begin{center}
\includegraphics[width=0.5\textwidth]{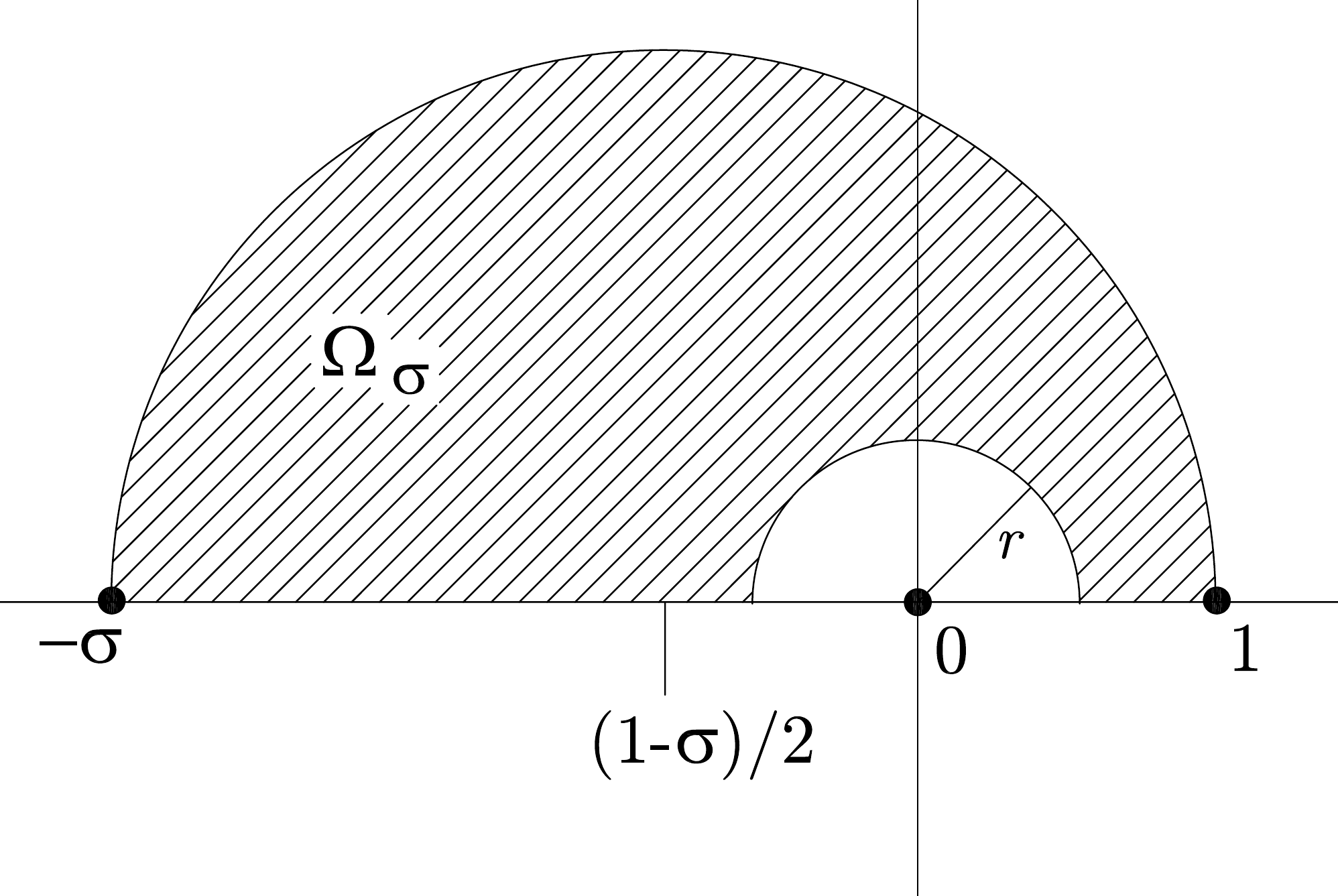}
\caption{The domain $\Omega_\sigma$.}
\label{fig:omega}
\end{center}
\end{figure}

The above arguments lead to the conformal parameterization
$X^{\sigma}=(x_1^{\sigma},x_2^{\sigma},x_3^{\sigma})\colon \Omega_\sigma  \to  \R ^3$ given by
 \begin{eqnarray*}
x_1^{\sigma}(z) & = & \mbox{Re} \left( \int_1^z \frac 12
\left( \frac{\sqrt{\sigma}}{u}-\frac{u}{\sqrt{\sigma}} \right) \frac{d \, u}{\sqrt{u (u-1) (u+\sigma)}} \right) ,\\
x_2^{\sigma}(z) & = & \mbox{Re} \left( \int_1^z \frac{i}2
\left( \frac{\sqrt{\sigma}}{u}+\frac{u}{\sqrt{\sigma}} \right)
\frac{d \, u}{\sqrt{u (u-1) (u+\sigma)}} \right) ,\\
x_3^{\sigma}(z) & = & \mbox{Re} \left( \int_1^z  \frac{d \, u}
{\sqrt{u (u-1) (u+\sigma)}} \right).
\end{eqnarray*}
The following properties are easy to check from the symmetries of the Weierstrass data:
\begin{enumerate}[(P1)]
\item The image through $X^{\sigma}$ of the boundary segment $[-\sigma,0]\cap \partial \Omega_\sigma $
corresponds to a planar geodesic of reflective symmetry of $M_{\sigma }$,
contained in the plane $\{x_2=0\}$.
\item The image through $X^{\sigma}$ of $[0,1]\cap \partial \Omega_\sigma$
is a straight line segment contained in the
surface and parallel to the $x_2$-axis.
\item The image through $X^{\sigma}$ of the outer half-circle in $\partial \Omega _{\sigma }$ is
a curve in $M_{\sigma }$ which is invariant
under the $180^\circ$-rotation around a straight line parallel to the $x_2$-axis, that passes
through the fixed point $X^{\sigma}(i\sqrt{\sigma})$ of $S_1$.
\end{enumerate}
\begin{figure}
\begin{center}
\includegraphics[width=7.5cm]{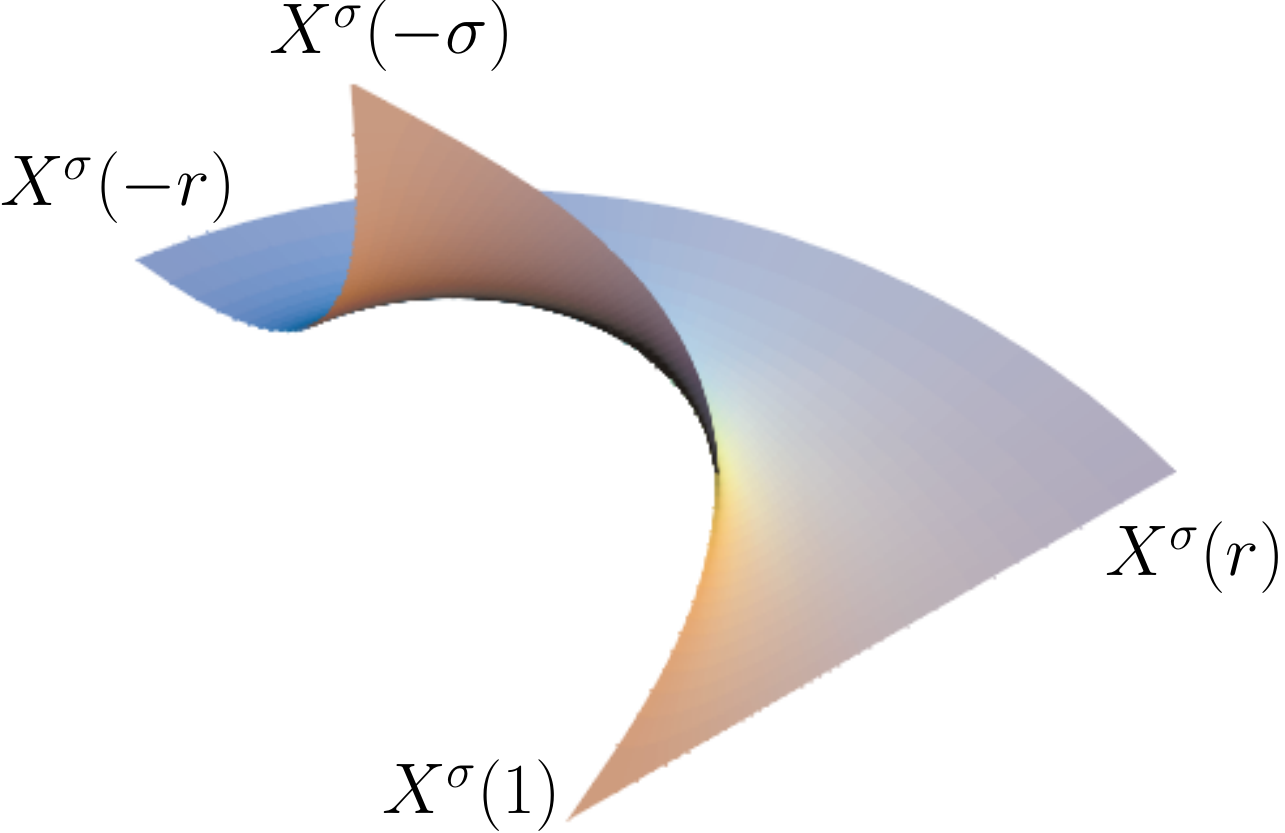}
\caption{The fundamental piece for $\sigma=1$.}
\label{fig:pieza-1}
\end{center}
\end{figure}
After reflecting the fundamental piece $X^{\sigma }(\Omega _{\sigma })$ across its boundary,
we will obtain the complete minimal surface
$M_{\sigma }\subset \R ^3$.

\subsection{Relationships between $M_\sigma$ and $R_\l$: the Shiffman function.}
\label{subsec:shiffman}

Each of the minimal surfaces $M_\sigma\subset \R^3$, $\sigma >0$, constructed
in the last section is topologically
a cylinder punctured in an infinite discrete set of points, and $M_{\sigma }$
has no points with horizontal tangent plane.
For a minimal surface with these conditions, M. Shiffman
introduced in 1956 a function that expresses the variation of the curvature of
planar sections of the surface. More precisely,
around a point $p$ of a minimal surface $M$ with nonhorizontal tangent plane,
one can always pick a local conformal coordinate
$\xi=x+i\, y$ so that the height differential is $\phi_3=d \, \xi$.
The horizontal level curves $\{ x_3=c\} \cap M$
near $p$ are then parameterized as $\xi_c(y)=c+i\, y$. If we call $\kappa_c$
to the curvature of this level section as a planar curve,
then one can check that
\[
\kappa_c (y)=\left[ \frac{|g|}{1+|g|^2} \mbox{ Re}
\left( \frac{g'}{g} \right) \right] |_ {\xi=\xi_c(y)},
\]
where $g$ is the Gauss map of $M$ in this local coordinate, see Section~\ref{subsec:weierstrass}.
The Shiffman function is defined as
\[
S_M=\Lambda \frac{\partial \kappa_c}{\partial y},
\]
where $\Lambda >0$ is the smooth function so that the induced metric on $M$ is
$ds^2=\Lambda^2 |d \xi |^2$. In particular,
the vanishing of the Shiffman function is equivalent to the fact that $M$ is
foliated by arcs of circles or straight lines
in horizontal planes. Therefore, a way to show that the surface $M_{\sigma }$
defined in the last section coincides with one
of the Riemann minimal examples $R_{\l }$ is to verify that its Shiffman function
vanishes identically.

A direct computation shows that
\begin{equation}
\label{eq:Shiffman}
S_M=\mbox{Im} \left[ \frac{3}{2} \left( \frac{g'}{g} \right)^2 -
\frac{g''}{g} - \frac{1}{1+|g|^2}
 \left( \frac{g'}{g} \right)^2 \right],
 \end{equation}
For the surface $M_{\sigma }$, we have $g(z,w)=g^\sigma(z,w)\tfrac{z}{\sqrt{\sigma}}$
and $d \xi=\phi_3^\sigma=\frac{dz}{w}= \sqrt{\sigma} \frac{g' \, d \xi}{w}$;
hence $w=\sqrt{\sigma }g'$ and thus $ {{(g')^2}=   {{g\,\left( \sqrt{\sigma}  + g \right) \,
\left( -1 + \sqrt{\sigma} \,g \right) }}}$. Taking derivatives of this expression we obtain
${{g''}=
   {\frac{-\sqrt{\sigma} }{2} +
     \left( -1 + {\sqrt{\sigma} }^2 \right) \,g +
     \frac{3\,\sqrt{\sigma} \,{g}^2}{2}}} .  $ Plugging this in (\ref{eq:Shiffman}) we get
\[
S_M=\mbox{Im} \left[\frac{\left( \sigma -1 \right) \,
     \left( {|g|}^2-1 \right)  -
    4\,\sqrt{\sigma} \,\mbox{Re}(g)}{2\,
    \left( 1 + {|g|}^2 \right) } \right]=0,
\]
which implies that $M_\sigma$ is one of the Riemann minimal examples $R_\l$, but which one?

We must look for an expression of $\sigma >0$ in terms of $\l \in \R $ (or vice versa),
so that the surfaces $R_\l$ and $M_\sigma$ are congruent.
Since $g^\sigma(1,0)=1/\sqrt{\sigma},$ $g^\sigma(-\sigma,0)=-\sqrt{\sigma}$, and we know that $X^\sigma(1,0)$, $X^\sigma(-\sigma,0)$
are points where straight lines contained in $M_{\sigma }$ intersect the vertical plane of symmetry,
the values of the stereographically
projected Gauss map of the surface at these points will help us to find $\sigma(\l )$.
Recall that with the parameterization
$X_\l(q,v)$ of $R_{\l }$ given in equation
(\ref{eq:Riemann1}), these points are given by taking $v=0$ and $q\to \infty$.
Hence we must compute the limit
as $q\to \infty $ of $\frac{N_1(q,0)}{1-N_3(q,0)}$, where $N=(N_1,N_2,N_3)$
is the Gauss map associated to $X_\l$. This limit is
$\frac{2}{\l -\sqrt{\l^2+4}}<0$, so we must impose it to coincide with $-\sqrt{\sigma}$. In other words,
\[
\sigma=\left(\frac{2}{\sqrt{\l^2+4}-\l} \right)^2.
\]
%%%%%%%%%%%%%%%%%%%%%%%%%%%%%%%%%%%%%%%%%%%%%%%%%%

%%%%%%%%%%%%%%%%%%%%%%%%%%%%%%%%%%%%%%%%%%%%%%%%%%

\subsubsection{Parameterizing the surface with {\sf Mathematica}.}
\label{subsec:parametrizando}
When using \textsf{Mathematica}, we must keep in mind that the germs of multivalued functions
that appear in the integration of the Weierstrass representation are necessarily continuous.
These choices of branches do not always coincide with the choices made by the program, but we do not want
to bother the reader with these technicalities. Thus, we directly write the three coordinate functions of the parameterization
(already integrated) as follows:
\[
\begin{array}{lll}
\tt
x1[\sigma_{-}][z_{-}]&:=&\frac{1}{\sqrt{\sigma}} \left(
\left( \sqrt{\frac{-1+z}z} \sqrt{z+\sigma}+ \right. \right. \\
& & \tt \frac 1{\sqrt{1+\sigma}} \left(2 \left( (-1-\sigma) EllipticE \left[ArcSin\left[\sqrt{1+\frac z{\sigma}}\right],\frac{\sigma}{1+\sigma}\right]+ \right. \right.\\
& & \tt \left. \left. \left. \left. EllipticF\left[ArcSin\left[\sqrt{1+\frac z{\sigma}}\right],\frac{\sigma}{1+\sigma}\right]  \right) \right)\right) \right)
\\
\rule{0cm}{.6cm}
\tt
x2[\sigma_{-}][z_{-}]&:=& - {\sqrt{\frac{-1 + z}{-\sigma \,z}}}\,{\sqrt{\sigma + z}} \\
\rule{0cm}{.6cm}
\tt
x3[\sigma_{-}][z_{-}]&:=&\tt \frac{-2}{{\sqrt{\sigma}}} EllipticF \left[ArcSin \left[\frac{{\sqrt{\sigma}}}{{\sqrt{\sigma + z}}}\right],\frac{1 + \sigma}{\sigma}\right]
\end{array}
\]
We will translate the surface so that the point $v_0^\sigma=X^\sigma(1)$ equals the origin, defining
the parameterization $\psi^\sigma(z)=X^\sigma(z)-v_0^\sigma$:
\[
\begin{array}{lll}
\tt
 v0[\sigma_{-}]&:=& \tt \left\{ -2\,Im \left[\frac{1}{\sqrt{-\left( 1 +\sigma \right) \,\sigma}} \left(\left( -1  -\sigma \right) \, EllipticE \left[ ArcSin \left[\sqrt{\frac{1 + \sigma}{\sigma}} \right] ,\frac{\sigma}{1 + \sigma}\right] + \right. \right. \right. \\
& & \tt \left. \left. EllipticF \left[ ArcSin \left[\sqrt{\frac{1 + \sigma}{\sigma}} \right],\frac{\sigma}{1 + \sigma}\right] \right)\right],0,  \\
\tt
& &
\tt \left. -2\,{Re}\left[\frac{1}{\sqrt{\sigma}}
EllipticF \left[ ArcSin \left[{\sqrt{\frac{\sigma}{1 + \sigma}}}\right],\frac{1 + {\sigma}}{{\sigma}}\right]
\right] \right\}
\\
\rule{0cm}{.6cm}
\tt \psi[\sigma_{-}][z_{-}]&:=& \tt Re[\{x1[\sigma][z],x2[\sigma][z],x3[\sigma][z]\}]-v0[\sigma]
\end{array}
\]

In order to simplify our parameterization, we will use a M\"{o}bius transformation that maps the half-annulus
$\{ e \leq |z| \leq 1, \; \mbox{Im}(z) \geq 0 \} $, for certain $e\in (0,1)$,
into the domain region $\Omega_\sigma$;
we will also use polar coordinates in the half-annulus:
\[
\begin{array}{lll}
\tt
f[z_{-}, a_{-},e_{-}] &:=& \tt
(-a^2 (1 + e^2) (-1 + z) -
          2 a (-3 + e^2) z - (1 + e^2) (1 +
                z) + \\
& & \tt Sqrt[4 (-1 + a)^2 e^2 + (1 + a)^2 (-1 + e^2)^2] (1 + a (-1 + z) + z)) \\
& & \tt (2 \, Sqrt[4 (-1 + a)^2 e^2 + (1 + a)^2 (-1 + e^2)^2] \\
& & \tt -2 ((1 + a) (-1 + e^2) - 2 (-1 + a) z))
\end{array}
\]
The graphics representation of the fundamental piece $\Omega _{\sigma }$
through the immersion $\psi ^{\sigma }$ is given by the following expression; observe
that we leave as variables the parameter $e\in (0,1)$, which measures how much of the
end of $\Omega _{\sigma }$ is represented (the smaller the value of $e$, the larger size of
the represented end) and the parameter $\sigma $ of the minimal Riemann example.
\[
\begin{array}{lll}
\tt r1[\sigma_{-}][e_{-}] &:=& \tt  ParametricPlot3D[
      \psi[\sigma][f[r Exp[I \;Pi \;t], \sigma, e]], \{r, e, 1 \}, \{t, 0, 1\}, \\
& &  \tt    PlotPoints -> \{40, 60 \}, Axes -> False];
\end{array}
\]

We are now ready to render the fundamental piece that appears in
Figure~\ref{fig:pieza-1}, which corresponds to execute the command
$\tt r1[2][0.1]$ (i.e., $\sigma =2$, $e=0.1$). We type:
\[
\tt p1=r1[2][0.1]
\]
\par
\vspace{-.4cm}
\begin{center}
\includegraphics[width=0.3\textwidth]{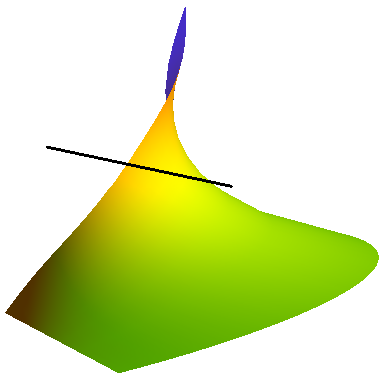}
\end{center}
In the last figure, we have also represented a straight line parallel to the $x_2$-axis,
that intersects the surface orthogonally at a point in the boundary of the fundamental piece;
the command to render both graphics objects simultaneously is
\[
\tt p2=Show[p1,line]
\]
where
\[
\begin{array}{lll}
\tt line&=& \tt ParametricPlot3D[\psi [2][Sqrt[2]I]+t\{ 0,1,0\} ,\{ t,-2,2\} ,\\
& & \tt Axes->False,\ Boxed->False,\ \\
& & \tt PlotRange->All,\ PlotStyle->Thickness[0.005]];
\end{array}
\]
produces the line $\psi ^{\sigma }(\sqrt{2}i)+\mbox{Span}(0,1,0)$. Next
we will extend the fundamental piece by $180^\circ$-rotation around this line, which induces
the holomorphic involution $S_1$ explained in Section~\ref{sec4.1.1}. To define this transformation
of the graphic, we will use the command ${\tt GeometricTransformation[x,\{ m,w\} ]}$ that
applies to a graphics object {\tt x} the affine transformation $x\mapsto mx+w$
(here $m$ is a real $3\times 3$ matrix and
$w\in \R^3$ a translation vector). In our case,
\[
m=\left( \begin{array}{ccc}
-1 & 0 & 0 \\
0 & 1 & 0 \\
0 & 0 & -1\end{array}\right) ,\qquad w=(2c_1,0,2c_3),
\]
where $c=\psi ^{\sigma}(\sqrt{2}i)$. We type:
\[
\begin{array}{lll}
\tt p3&=& \tt Graphics3D[GeometricTransformation[p1[[1]], \\
& & \tt \{ \{ \{ -1,0,0\} ,\{ 0,1,0\} ,\{ 0,0,-1\} \} ,\{ 2c1,0,2c3\} \} ]];
\end{array}
\]
after having defined {\tt c1} and {\tt c3} as the first and third coordinates
of $\psi ^{\sigma }(\sqrt{2}i)$. In order to render the pieces {\tt p1} and {\tt p3} at the
same time, as well as the $180^\circ$-rotation axis, we type:
\[
\tt p4=Show[p1,p3,line]
\]
\begin{center}
\includegraphics[width=0.4\textwidth]{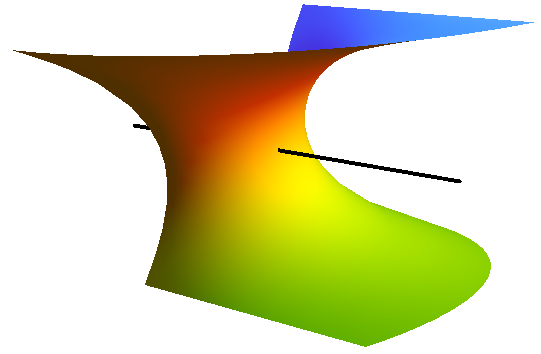}
\end{center}

The next step consists of reflecting the last figure in the plane $\{ x_2=0\} $ (which is the plane orthogonal
to the line segments contained in the boundary of the last piece and to the orientation-preserving
$180^\circ$-rotation axis).
\[
\begin{array}{lll}
\tt p5&=& \tt Graphics3D[GeometricTransformation[p4[[1]], \\
& & \tt \{ \{ \{ 1,0,0\} ,\{ 0,-1,0\} ,\{ 0,0,1\} \} \} ]]; \\
\tt p6&=& \tt Show[p4,p5,line]
\end{array}
\]
\begin{center}
\includegraphics[width=0.4\textwidth]{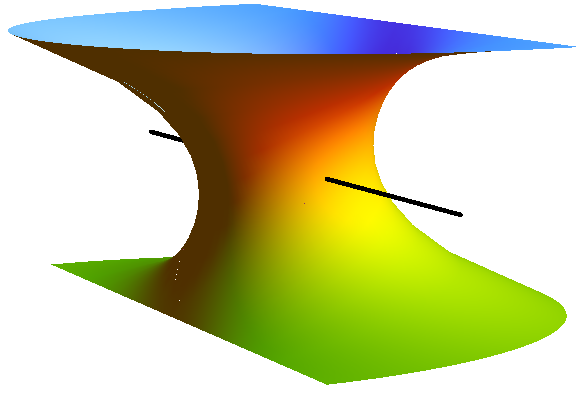}
\end{center}
Next we rotate the last piece by angle $180^\circ$ around the $x_2$-axis, which is contained
in the boundary of the surface.
\[
\begin{array}{lll}
\tt p7&=& \tt Graphics3D[GeometricTransformation[p6[[1]], \\
& & \tt \{ \{ \{ -1,0,0\} ,\{ 0,1,0\} ,\{ 0,0,-1\} \} \} ]]; \\
\tt p8&=& \tt Show[p6,p7]
\end{array}
\]
\begin{center}
\includegraphics[width=0.55\textwidth]{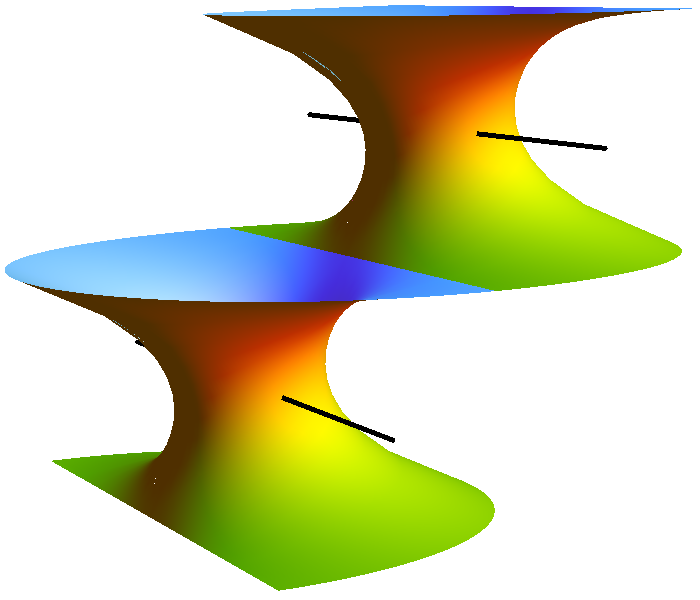}
\end{center}
$\tt p8$ represents a fundamental domain of the Riemann minimal example.
The whole surface can be now obtained by translating the graphics domain $\tt p8$ by
multiples of the vector $2 \, t_0= 2\psi^{\sigma }(-\sigma)$. We define this translation vector $t_0$:
\[
\tt t0=\psi[2][-1.99999]
\]

The reason why we have evaluated at a point close to $-2$ on the right, is due to the
aforementioned fact that we must use a continuous branch
of the elliptic functions that appear in the expression of $\psi$. The numeric error that
we are making is insignificant with a normal
screen resolution. After this translation, we type:
\[
\begin{array}{lll}
\tt p9&=& \tt Graphics3D[GeometricTransformation[p8[[1]], \\
& & \tt \{ \{ \{ 1,0,0\} ,\{ 0,1,0\} ,\{ 0,0,1\} \} ,2\ t0\} ]]; \\
\tt p10&=& \tt Show[p8,p9]
\end{array}
\]
\begin{center}
\includegraphics[width=0.5\textwidth]{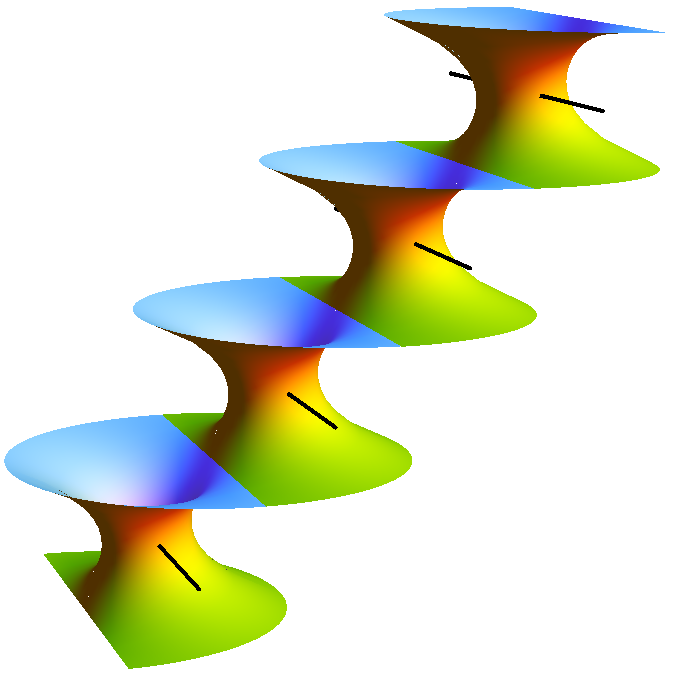}
\end{center}

It is desirable to have a better understanding of the surface ``at infinity''.
This can done by taking a smaller value of the parameter {\tt e} y repeating the whole
process. Figure~\ref{R-image} represents the final stage {\tt p10} in the case {\tt e=0.02}.
\begin{figure}[H]
\begin{center}
\includegraphics[width=0.6\textwidth]{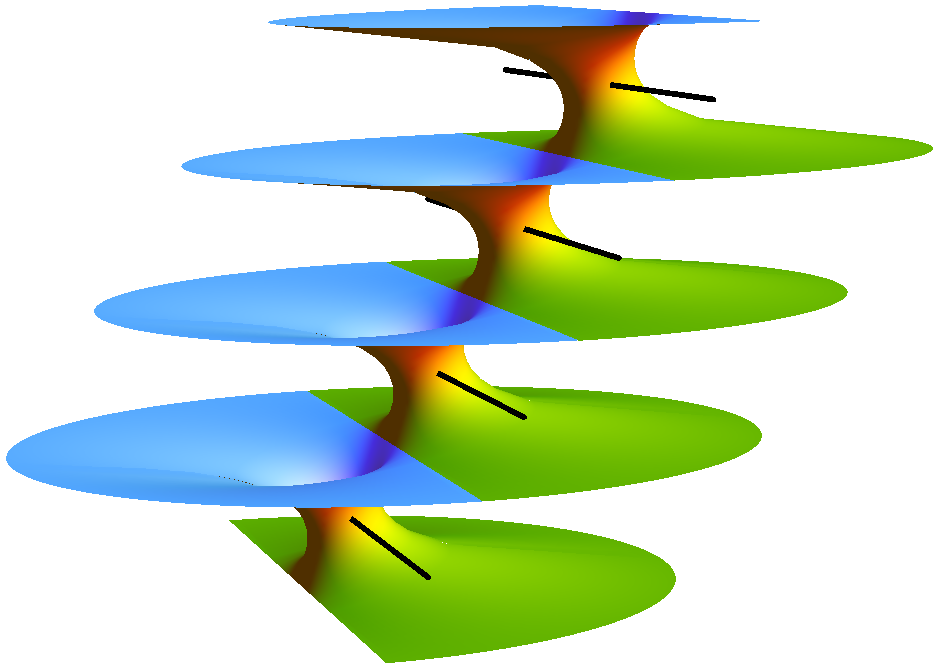}
\end{center}
\caption{A image  of one of the Riemann minimal examples.}
\label{R-image}
\end{figure}
Figure~\ref{R-image} indicates that the surface becomes asymptotic to an infinite family of parallel
(actually horizontal) planes, equally spaced.
This justifies the wording {\em planar ends.}
\begin{remark}
{\rm
For very large or very small values of the parameter $\sigma $,
one can find imperfections in the graphics, especially around
$\psi^\sigma(1)$ or $\psi^\sigma(-\sigma)$.
This is due to the fact that these values of $\sigma $ produce a non-homogeneous
distribution of points in the mesh that the program computes when rendering the figure.
This issue can be solved by
substituting $\tt Exp[I \;Pi \;t]$ by $\tt Exp[I \;Pi \;t^{n}]$
in the definition of $\tt r1[\sigma][z]$, with $\tt n$ large if
$\sigma$ is close to zero $0$, or with $\tt n$ close to zero if $\sigma$ is large.
}
\end{remark}

\section{Uniqueness of the properly embedded minimal planar domains in $\R^3$.}
\label{MPR}

As mentioned in Section~\ref{riemann}, each Riemann minimal example $R_{\l }$
is a complete (in fact, proper)
embedded minimal surface  in $\R^3$ with the topology of a cylinder punctured
in an infinite discrete set of points,
which is invariant under a translation. If we view this cylinder as a twice punctured
sphere, then one deduces that $R_{\l }$ is topologically (even conformally) equivalent to a
sphere minus an infinite set of points that accumulate only at distinct two points, the so called limit
ends of $R_{\l }$. In particular, $R_{\l }$ is a planar domain. One longstanding open problem
in minimal surface theory has been the following one:
\begin{quote}
{\it
Problem: Classify all properly embedded minimal planar domains in $\R^3$.
}
\end{quote}
Up to scaling and rigid motions, the family ${\cal P}$ of all properly embedded planar domains in $\R^3$
comprises the plane, the catenoid, the helicoid and the 1-parameter family of Riemann minimal
examples. The proof  that this list is complete is a joint effort described in a series of papers by
different researchers. In this section we will give an outline of this classification.

\subsection{The case of finitely many ends, more than one.}
\label{sec5.1}
One starts by considering a surface $M\in {\cal P}$ with $k$ ends, $2\leq k<\infty $.
Even without
assuming genus zero, such surfaces of finite genus were proven to have finite total curvature
by Collin~\cite{col1}, a case in which
the asymptotic geometry is particularly well-understood: the finitely many ends
are asymptotic to planes or
half-catenoids, the conformal structure of the surface is the one of a compact Riemann
surface minus a finite
number of points, and the Gauss map of the surface extends meromorphically
to this compactification
(Huber~\cite{hu1}, Osserman~\cite{os1}). The asymptotic behavior of the ends
and the embeddedness of $M$ forces
the values of the extended Gauss map at the ends to be contained in a set of
two antipodal points on the sphere.
After possibly a rotation in $\R^3$, these values of the extension of the
Gauss map of $M$ at the ends can be assumed to be $(0,0,\pm 1)$.
L\'opez and Ros~\cite{lor1} found the following
argument to conclude that $M$ is either a plane or a catenoid.
The genus zero assumption  plays a
crucial role in the well-posedness of the deformation $\l >0\mapsto M_{\l }$ of $M$
by minimal surfaces $M_{\l }
\subset \R^3$ with the same conformal structure and the same height differential
as $M$ but with the meromorphic Gauss map
scaled by $\l\in [1,\infty)$. A clever application of the maximum
principle for minimal surfaces along this deformation implies that all surfaces
$M_{\l }$ are embedded and that
if $M$ is not a plane, then $M$ has no points with horizontal tangent planes and no planar ends
(because either of these cases produces self-intersections in $M_{\l }$ for sufficiently large
or  small values of $\l >0$);
in this situation, the third coordinate
function of $M$ is proper without critical points, and an application of Morse theory implies that $M$
has the topology of an annulus. From here it is not difficult to prove that $M$ is a catenoid.

\subsection{The case of just one end.}
\label{sec5.2}
The next case to consider is when $M\in {\cal P}$ has exactly one end, in particular $M$
is topologically a plane, and
the goal is to show that $M$ congruent to a plane or to a helicoid.
This case was solved by Meeks and Rosenberg~\cite{mr8},
by an impressive application of a series of powerful new tools in minimal surface theory:
the study of minimal laminations, the so-called Colding-Minicozzi theory and some progresses in the
understanding of the conformal structure of complete minimal surfaces. The first two of these tools
study the possible limits of a sequence of embedded minimal
surfaces under lack of either uniform local area bounds (minimal laminations) or
uniform local curvature bounds
(Colding-Minicozzi theory) or even both issues happening simultaneously; this is
in contrast to the classical situation for describing such limits, that requires
both uniform local area and curvature bounds to obtain a classical limit minimal surface.

The study of minimal laminations, carried out in~\cite{mr8} by Meeks and Rosenberg,
allows one to relate completeness of embedded minimal surfaces in $\R^3$ to their properness,
which is a stronger condition.
For instance, in~\cite{mr8} Meeks and Rosenberg
proved that if $M\subset \R^3$ is a connected, complete embedded minimal surface
with finite topology and bounded
Gaussian curvature on compact subdomains of $\rth$, then $M$ is proper; this properness conclusion
was later generalized to the case of finite genus
by Meeks, P\'erez and Ros in~\cite{mpr3}, and Colding and Minicozzi~\cite{cm35}
proved it in the stronger case that one drops the
local boundedness hypothesis for the Gaussian curvature. The theory of minimal laminations
has led to other interesting results by itself, see e.g., \cite{Kh1,kl1,me25,me30,mpr10,mpr11,mpr18,mr13}.

Regarding Colding-Minicozzi theory in its relation to the classification by Meeks and Rosenberg of the
plane and the helicoid as the only simply connected elements in ${\cal P}$, its main result, called the
{\it Limit Lamination Theorem for Disks,} describes
the limit of (a subsequence of) a sequence of compact, embedded minimal disks $M_n$ with boundaries
$\partial M_n$ lying in the boundary spheres of Euclidean balls $\B (R_n)$ centered at the origin, where the
radii of these balls diverge to $\infty $, provided that the Gaussian curvatures of the $M_n$ blow up at some
sequence of points in $M_n\cap \B (1)$: they proved that this limit is a foliation $\cF$ of $\R^3$ by parallel
planes, and the convergence of the $M_n$ to $\cF$ is of class $C^{\a }$, $\a \in (0,1)$, away from some
Lipschitz curve $S$ (called the singular set of convergence of the $M_n$ to $\cF$)
that intersects each of the planar leaves of $\cF$ transversely just once,
and arbitrarily close to every point of $S$,
the Gaussian curvature of the $M_n$ also blows up as $n\to \infty $, see~\cite{cm23}
for further details. There is another
result of fundamental importance in Colding-Minicozzi theory, that is used to prove
the Limit Lamination Theorem for Disks
and also to demonstrate the results in~\cite{cm35,mpr3,mr8} mentioned in the last paragraph,
which is called
the {\it one-sided curvature estimate}~\cite{cm23}, a scale invariant bound for the
Gaussian curvature
of any embedded minimal disk in a half-space.

With all these ingredients in mind, we  next give a rough sketch of the proof by Meeks and Rosenberg
of the uniqueness of the helicoid as the unique simply connected, properly embedded,
nonplanar minimal surface
in $\R^3$. Let $M\in {\cal P}$ be a simply-connected surface. Consider any sequence
of positive numbers $\{ \lambda_n \}_{n}$ which
decays to zero and let $\lambda_n M$ be the surface scaled
 by $\lambda_n$.  By the Limit Lamination Theorem for Disks, a subsequence of
 these surfaces converges on
 compact subsets of $\rth$ to a minimal foliation $\cF$ of
 $\rth$ by parallel planes, with singular set of convergence $S$
being a Lipschitz curve that can be parameterized by the height over
the planes in $\cF $. An application of Colding-Minicozzi theory assures that the
limit foliation $\cF$
is independent of the sequence $\{\lambda_n\}_n$. After a rotation of $M$ and replacement of the $\l _nM$
 by a subsequence, one can suppose that the $\l _nM$ converge to the foliation $\cF$ of $\rth$ by
horizontal planes, outside of the singular set of convergence given by a Lipschitz curve $S$
parameterized by its
$x_3$-coordinate. A more careful analysis of the convergence of $\l _nM$ to $\cF$ allows also to show that
$M$ intersects transversely each of the planes in $\cF$. This implies that the Gauss map of $M$
does not take vertical values, so after composing with the stereographical projection we can
write this Gauss map $g \colon M \to \C \cup \{\infty \}$ as
\begin{equation}
\label{eq:ghelic}
g(z)=e^{f(z)}
\end{equation}
for some holomorphic function $f\colon M \to \C$, and the height differential $\phi _3$ of $M$
has no zeros or poles. The next step in the proof is to check that the conformal structure
of $M$ is $\C $; to see this, first observe that the nonexistence of points in $M$ with vertical
normal vector implies that the intrinsic gradient of the third coordinate function $x_3\colon M\to \R $
has no zeros on $M$. In a delicate argument that uses both the above Colding-Minicozzi picture for
limits under shrinkings of $M$ and a finiteness result for the number of components of a minimal graph over
a possibly disconnected, proper domain in $\R^2$ with zero boundary values, Meeks and Rosenberg proved that
none of the integral curves of $\nabla x_3$ is asymptotic to a plane in $\cF$, and that
every such horizontal plane intersects $M$ transversely in a single proper arc. This is enough to
use the conjugate harmonic function $x_3^*$ of $x_3$ (which is well-defined on $M$ as the surface
is simply connected) to show that $x_3+ix_3^*\colon M\to \C $ is a conformal diffeomorphism.
Once one knows that $M$ is conformally $\C $, then we can reparameterize conformally $M$ so that
\[
\phi _3=dx_3 + i\, dx_3^* =dz;
\]
in particular, the third coordinate $x_3 \colon \C \to \R$ is $x_3(z)=\mbox{Re} (z)$.

To finish the proof,
it only remains to determine the Gauss  map $g$ of $M$, of which we now have the description
(\ref{eq:ghelic}) with $f\colon \C \to \C $ entire.
If the holomorphic function  $f(z)$ is a linear function of the form $az+b$, then one deduces that
$M$ is an associate surface\footnote{The family of {\it associate
surfaces} of a simply connected minimal surface with Weierstrass
data $(g,\phi _3)$ are those with the same Gauss map and height
 differential $e^{i\t }\phi _3$, $\t \in [0,2\pi )$. In particular, the case $\t =\pi /2$
 is the conjugate surface. This notion can be generalized to non-simply connected surfaces,
 although in that case the associate surfaces may have periods.} to the helicoid; but none of the
nontrivial associate surfaces to the helicoid are injective as mappings, which implies that $M$ must be the
helicoid itself when $f(z)$ is linear. Thus, the proof reduces to proving that $f(z)$ is linear.
The explicit expression for the Gaussian curvature $K$ of $M$ in terms of $g,\phi _3$ is
\begin{equation}
\label{eq:K}
K=-\left( \frac{4\left| dg/g\right| }{(|g|+
|g|^{-1})^2|\phi _3|}\right) ^2.
\end{equation}
Plugging our formulas (\ref{eq:ghelic}), $\phi _3=dz$ in our setting, an application of
Picard's Theorem to $f$ shows that
$f(z)$ is  linear if and only if $M$ has bounded curvature. This boundedness curvature assumption for $M$ can be achieved
by a clever blow-up argument, thereby finishing the sketch of the proof.  For further
 details, see~\cite{mr8}.

\subsection{Back to the Riemann minimal examples: the case of infinitely many ends.}
To finish our outline of the solution of the classification problem stated at
the beginning of Section~\ref{MPR}, we must explain how
to prove that the Riemann minimal examples are the only properly embedded planar domains in $\R^3$
with infinitely many ends.

Let $M\in {\cal P}$ be a surface with infinitely many ends. First we analyze the structure of
the space ${\cal E}(M)$ of ends of $M$. ${\cal E}(M)$ is the quotient
${\mathcal A}/_\sim $ of the set
\[
{\mathcal A}=\{ \a \colon [0,\infty )\to M\ | \ \a \mbox{ is a proper arc}\}
 \]
(observe that ${\mathcal A}$ is nonempty as $M$ is
not compact) under the following equivalence relation: Given $\a_1, \a_2\in {\mathcal A}$, $\a _1\sim \a _2$
if for every compact set $C \subset M$, there exists $t_C\in [0,\infty )$ such that $\a _1(t),\a _2(t)$
lie the same component of $M-C$, for all $t\geq t_C$. Each equivalence class in ${\mathcal E}(M)$ is called
a {\it topological end} of $M$. If $e\in {\mathcal E}(M)$, $\a \in e$ is a representative proper arc
and $\Omega \subset M$ is a proper subdomain with  compact boundary such that $\alpha \subset \Omega $,
then we say that the domain $\Omega $ {\it represents} the end $e$.

The space ${\mathcal E}(M)$ has a natural topology, which is defined in terms of a basis of open sets:
for each proper domain $\Omega \subset M$ with compact boundary, we define the basis open
set $B(\Omega ) \subset {\mathcal E} (M)$ to be those equivalence
classes in ${\mathcal E}(M)$ which have representatives contained in
$\Omega $. With this topology, ${\mathcal E}(M)$ is a totally disconnected compact Hausdorff space
which embeds topologically as a subspace of $[0,1]\subset \R $ (see e.g., pages 288-289 of \cite{mpe1}
 for a proof of this embedding result for ${\mathcal E}(M)$).

In the particular case that $M$ is a properly embedded minimal surface in $\R^3$ with more than one end,
a fundamental result is that ${\mathcal E}(M)$ admits a geometrical ordering by relative heights over
a plane called the {\it limit plane at infinity} of $M$. To define this reference plane, Callahan,
Hoffman and Meeks~\cite{chm3} showed that in one of the closed complements
 of $M$ in $\R^3$, there exists a noncompact, properly embedded minimal surface
$\Sigma $ with compact boundary and finite total curvature. The ends of $\Sigma $ are then
of catenoidal or planar type, and the embeddedness of
$\Sigma $ forces its ends to have parallel normal vectors at infinity.
The limit tangent plane at infinity of $M$ is the
plane in $\R^3$ passing through the origin, whose normal
vector equals (up to sign) the limiting normal vector at the ends of $\Sigma $.
It can be proved that such a plane does not depend on the finite total curvature
 minimal surface $\Sigma \subset \R^3-M$~\cite{chm3}. With this notion in hand,
the ordering theorem is stated as follows.

\begin{theorem} {\bf (Ordering Theorem,  Frohman, Meeks~\cite{fme2})}
\label{thordering}
Let $M\subset \R^3$ be a properly embedded minimal surface with
more than one end and horizontal limit tangent plane at
infinity. Then, the space ${\mathcal E}(M)$ of ends of
$M$ is linearly ordered geometrically by the relative
heights of the ends over the
$(x_1,x_2)$-plane, and embeds topologically in $[0,1]$
in an ordering preserving way. Furthermore,
this ordering has a topological nature in the following
sense: If $M$ is properly isotopic to a properly embedded
minimal surface $M'$ with
horizontal limit tangent plane at infinity, then the associated
 ordering of the ends
of $M'$ either agrees with or is opposite to the ordering coming
from $M$.
\end{theorem}

Given a minimal surface $M\subset \R^3$ satisfying the hypotheses of
Theorem~\ref{thordering}, we define the {\it top end} $e_T$ of $M$ as the unique
maximal element in  ${\mathcal E}(M)$ for the ordering
given in this theorem (as ${\mathcal E}(M)\subset [0,1]$
is compact, then $e_T$ exists). Analogously, the {\it bottom
end} $e_B$ of $M$ is the unique minimal element in ${\mathcal E}(M)$.  If
$e\in {\mathcal E}(M)$ is neither the top nor the bottom end of $M$,
then it is called a {\it middle end} of $M$. There is another way of
grouping ends of such a surface $M$ into simple and limit ends; for
the sake of simplicity and as we are interested in discussing the classification
of surfaces $M\in {\cal P}$, we will restrict in the sequel to the case of
a surface $M$ of genus zero.

Given $M\in {\cal P}$, an isolated point $e \in {\mathcal E} (M)$ is
 called a {\it simple end of $M$}, and $e$ can be represented
by a proper subdomain $\Omega \subset M$ with compact boundary which is homeomorphic
to the annulus $\esf ^1 \times [0,\infty)$. Because of this model, $e$ is also
called an {\it annular end.} On the contrary, ends in ${\mathcal E} (M)$ which are
 not simple (i.e., they are limit points of ${\mathcal E}(M)\subset [0,1]$) are called
{\it limit ends} of $M$. In our situation of $M$ being a planar domain, its
limit ends can be represented by proper subdomains $\Omega  \subset M$ with
compact boundary, genus zero and infinitely many ends. As in this section
$M$ is assumed to have
infinitely many ends and ${\mathcal E}(M)$ is compact, then $M$ must have at least
one limit end.

Each of the planar ends of a Riemann minimal
example $R_{\l }$ is a simple annular middle end, and $R_{\l }$ has two limit ends corresponding
to the limits of planar ends as the height function of $R_{\l }$ goes to $\infty $ (this
is the top end of $R_{\l }$) or to $-\infty $ (bottom end). Thus, middle ends of $R_{\l }$
correspond to simple ends, and its top and bottom ends are limit ends.
Most of this behavior is in fact valid for any properly embedded minimal surface
$M\subset \R^3$ with more than one end:

\begin{theorem}[Collin, Kusner, Meeks, Rosenberg~\cite{ckmr1}]
\label{thmckmr}
 Let $M\subset \R^3$ be a properly embedded minimal surface with more
than one end and horizontal limit tangent plane at infinity. Then, any
limit end of $M$ must be a top or bottom end of $M$. In particular, $M$ can have
 at most two limit ends, each middle end is simple and the number of ends of $M$ is
countable.
\end{theorem}

In the sequel, we will assume that our surface $M\in {\cal P}$ has horizontal limit
tangent plane at infinity. By Theorem~\ref{thmckmr}, $M$ has no middle
limit ends, hence either it has one limit end (this one being its top
 or its bottom limit end) or both top and bottom ends are the limit ends of $M$, like
 in a Riemann minimal example. The next step in our description of the
 classification of surfaces in ${\cal P}$ is due to
Meeks, P\'erez and Ros~\cite{mpr4}, who discarded
the one limit end case through the following result.

\begin{theorem}[Meeks, P\'erez, Ros \cite{mpr4}]
\label{thmno1limitend}
 If $M\subset \R^3$ is a properly embedded minimal surface
with finite genus, then $M$ cannot have exactly one limit end.
\end{theorem}
The proof of Theorem~\ref{thmno1limitend} is by contradiction. One assumes
that the set of ends of $M$, linearly ordered by increasing heights by
the Ordering Theorem~\ref{thordering}, is ${\mathcal E}(M)=
\{ e_1,e_2,\ldots ,e_{\infty }\} $ with the limit end of $M$ being its
 top end $e_{\infty }$. Each annular end of $M$ is a simple end and
can be proven to be asymptotic to a graphical annular end $E_n$ of a vertical
catenoid with negative logarithmic growth $a_n$ satisfying $a_1 \leq \ldots \leq a_n \leq \ldots < 0$
(Theorem~2 in Meeks, P\'erez and Ros~\cite{mpr3}). Then one studies
the subsequential limits of homothetic shrinkings $\{ \l _nM\}
_n$,  where  $\{ \l _n\} _n\subset \R ^+$ is  any sequence of numbers decaying to zero;
recall that this was also a crucial step in the proof of the uniqueness of the
helicoid sketched in Section~\ref{sec5.2}. Nevertheless, the situation now is
more delicate as the surfaces $\l _nM$ are not simply connected. Instead, it can be
proved that the sequence $\{ \l _nM\} _n$ is {\it locally simply connected} in
$\R^3-\{ \vec{0}\} $, in the sense that given any point $p\in \R^3-\{ \vec{0}\} $,
there exists a number $r(p)\in (0,|p|)$ such that the open ball $\B (p,r(p))$
centered at $p$ with radius $r(p)$ intersects $\l _nM$ in compact disks whose boundaries lie
on $\partial \B (p,r(p))$, for all $n\in \N$. This is a
difficult technical part of the proof, where the
Colding-Minicozzi theory again plays a crucial role. Then one uses this locally simply connected property
in $\R^3-\{ 0\} $ to show that the
limits of subsequences of $\{ \l _nM\} _n$ consist of minimal laminations
${\mathcal L}$ of $H(*)=\{ x_3\geq 0\}
-\{ \vec{0}\} \subset \R^3$ containing $\partial H(*)$ as a leaf, and that the
singular set of convergence of the of $\l _nM$ to ${\mathcal L}$ is empty; from here
one has that
\begin{quote}
$(\star )$ The sequence $\{ |K_{\l _nM}|\} _n$ of absolute Gaussian curvature functions of the
$\l _nM$, is locally bounded in $\R^3-\{ 0\} $.
\end{quote}
In particular, taking $\l _n=|p_n|^{-1}$ where
$p_n$ is any divergent sequence of points on $M$, $(\star )$ implies that the absolute Gaussian
curvature of $M$ decays at least quadratically in terms of the distance function $|p_n|$ to
 the origin. In this setting, the Quadratic Curvature Decay Theorem stated in
Theorem~\ref{thm1introd} below implies that $M$ has finite total curvature; this is impossible in our
situation with infinitely many ends, which finishes our sketch of proof of Theorem~\ref{thmno1limitend}.

\begin{theorem} {\bf (Quadratic Curvature Decay Theorem, \, Meeks, P\'erez,
Ros~\cite{mpr10})} \label{thm1introd} Let $M\subset \R^3-\{
\vec{0}\} $ be an embedded minimal surface with compact boundary
(possibly empty), which is complete outside the origin $\vec{0}$;
i.e., all divergent paths of finite length on~$M$ limit to
$\vec{0}$. Then, $M$
%A complete, embedded minimal
%surface in $\R^3$ with compact boundary
%   (possibly empty)
has quadratic decay of curvature\footnote{This means that $|K_M|R^2$
is bounded on $M$, where $R^2=x_1^2+x_2^2+x_3^2$.}
if and only if its closure in $\R^3$ has finite total curvature.
\end{theorem}

Once we have discarded the case of a surface $M\in {\cal P}$ with just one limit
end, it remains to prove that when $M$ has two limit ends,
% (which after assuming that $M$
%has horizontal limit tangent plane at infinity,  must be
%its top and bottom ends by Theorem~\ref{thmckmr} above),
then $M$ is a
Riemann minimal example. The argument for proving this is also delicate,
but since it uses strongly the Shiffman function and its surprising connection
to the theory of integrable systems and more precisely, to the Korteweg-de Vries equation,
we will include some details of it.

We first need to establish a framework for $M$ which makes possible to use globally the
Shiffman function; here the word {\it globally} also refers its {\it extension across the planar ends} of
$M$, in a strong sense to be precise soon. Recall that we have normalized $M$ so that
its limit tangent plane at infinity is the $(x_1,x_2)$-plane.
By Theorem~\ref{thmckmr},
the middle ends of $M$ are not limit ends, and as $M$ has genus zero, then these middle ends
can be represented by annuli. Since $M$ has more than one end, then every annular end of $M$ has
finite total curvature (by Collin's theorem~\cite{col1}, that we also used at the beginning
of Section~\ref{sec5.1}), and thus such annular ends of $M$ are asymptotic to the ends of planes or catenoids.
Now recall Theorem~\ref{thmckmr} above, due to Collin, Kusner, Meeks and Rosenberg. The same authors
obtained in~\cite{ckmr1} the following additional information about the middle ends:

\begin{theorem}[Theorem~3.5 in~\cite{ckmr1}]
\label{thm5.5}
Suppose a properly embedded minimal surface $\Sigma$ in $\rth$ has two limit ends with horizontal limit tangent
plane at infinity.
Then there exists a sequence of horizontal planes $\{ P_j\} _{j\in \N }$ in $\R^3$ with
increasing heights, such that $\Sigma$ intersects each $P_j$ transversely in a compact set,
every middle end of $\Sigma$ has an end representative which is the closure of the intersection of $\Sigma$ with
the slab bounded by $P_j\cup P_{j+1}$, and every such slab contains exactly one of these middle end representatives.
\end{theorem}

Theorem~\ref{thm5.5} gives a way of separating the middle ends of $M$ into regions determined by
horizontal slabs, in a similar manner as
the planar ends of a Riemann minimal example
can be separated by slabs bounded by horizontal planes.
Furthermore, the Half-space Theorem~\cite{hm10} by Hoffman and Meeks
ensures that the restriction of the third coordinate function $x_3$ to
the portion $M(+)$ of $M$ above $P_0$ is not bounded from above and extends smoothly across the middle ends.
Another crucial result, Theorem~3.1 in~\cite{ckmr1}, implies that $M(+)$ is conformally
parabolic (in the sense that Brownian motion on $M(+)$ is recurrent), in particular the annular simple
middle ends of $M$ in $M(+)$ are conformally punctured disks. After compactification of $M(+)$ by adding
its middle ends and their limit point $p_{\infty }$ corresponding to the top end of $M$,
we obtain a conformal parameterization of this compactification defined on the unit disk $\D =\{ |z|\leq 1\} $,
so that $p_{\infty }=0$, the middle ends
in $M(+)$ correspond to a sequence of points $p_j\in \D -\{ 0\} $ converging to zero as $j\to \infty $,
and
\[
x_3|_{M(+)}(z)=-\l \ln|z|+c
\]
for some $\l ,c\in \R $, $\l>0$. This implies that there are no points in $M(+)$ with horizontal tangent plane.
Observe that different planar ends in $M$
cannot have the same height above $P_0$ by Theorem~\ref{thm5.5}, which implies that $M(+)$ intersects
every plane $P'$ above $P_0$ in a simple closed curve
if the height of $P'$ does not correspond to the height of any middle end, while $P'$ intersects $M(+)$ is a
proper Jordan arc when the height of $P'$ equals the height of a middle end.
Similar reasoning can be made for the surface $M(-)=M-[M(+)\cup P_0]$.
From here one deduces easily that the meromorphic extension through the planar ends
of the stereographically projected Gauss map $g$ of $M$ has order-two zeros and poles
at the planar ends, and no other
zeros or poles in $M$. This is a sketch of the proof of the first four items of the
following descriptive result,
which is part of Theorem~1 in~\cite{mpr3}.

\begin{theorem}
\label{thm5.6}
Let $M\in {\cal P}$ be a surface with infinitely many ends. Then, after a rotation and a homothety we have:
\begin{enumerate}
\item $M$ can be conformally parameterized by the cylinder $\C /\langle
i\rangle $ punctured in an infinite
discrete set of points $\{ p_j,q_j\} _{j\in \Z }$
which correspond to the planar ends of $M$.

\item The stereographically projected Gauss map $g\colon (\C /\langle
i\rangle )-\{ p_j,q_j\} _{j\in \Z }\to \C \cup \{ \infty \} $
extends through the planar ends of $M$ to a meromorphic function
$g$ on $\C /\langle i\rangle $  which has double zeros at the points $p_j$ and double poles at the $q_j$.

\item The height differential of $M$ is $\phi _3=dz$ with $z$ being the usual conformal
coordinate on~$\C $, hence the third coordinate function of $M$ is $x_3(z)=\mbox{\rm Re} (z)$.

\item The planar ends of $M$ are ordered by their heights so that
$\mbox{\rm Re} (p_j)<\mbox{\rm Re} (q_j)<\mbox{\rm Re} (p_{j+1})$ for all $j$ with
$\mbox{\rm Re} (p_j)\to \infty $ (resp. $\mbox{\rm Re} (p_j) \to -\infty $)
 when $j\to \infty $ (resp. $j\to -\infty $).
\end{enumerate}
\end{theorem}

The description in Theorem~\ref{thm5.6} allows us to define globally the Shiffman function on
any surface $M$ as in that theorem.
To continue our study of properties of such a surface, we need the notion of {\it flux}.
The flux vector along a closed curve $\gamma \subset M$ is defined as
\begin{equation}
\label{eq:flux} F( \gamma) = \int_\gamma \mbox{Rot}_{90^\circ }(\g ')
= \mbox{Im} \int _{\g }\left( \frac{1}{2}\left(
\frac{1}{g}-g\right), \frac{i}{2}\left( \frac{1}{g}+g\right)
,1\right) \phi _3,
\end{equation}
where $(g,\phi _3)$ is the  Weierstrass data of $M$ and $\mbox{Rot}_{90^\circ }$ denotes
the rotation by angle $\pi /2$ in the tangent plane of $M$ at any point. It is easy to
show that $F(\g )$ only depends of the homology class of $\g $ in $M$, and that the
flux along a closed curve that encloses a planar end of $M$ is zero. In particular,
for a surface $M$ as in Theorem~\ref{thm5.6}, the only flux vector to consider is that associated
to any compact horizontal section $M\cap \{ x_3=\mbox{constant}\} $, which we will
denote by $F(M)$. Note that by item~(3) of Theorem~\ref{thm5.6}, the third component of $F(M)$ is 1.
In the sequel, we will assume the following normalization
for $M$ after possibly a rotation in $\R^3$ around the $x_3$-axis:
\begin{equation}
\label{star}
F(M)=(h,0,1) \ \mbox{ for some }h\geq 0.
\end{equation}
The next result collects some more subtle properties of a surface $M$ as in Theorem~\ref{thm5.6}
(this is the second part of Theorem~1 in~\cite{mpr3}).
\begin{theorem}
\label{thm5.7}
For a surface $M$ normalized as in (\ref{star}), we have:
\begin{enumerate}
 \setcounter{enumi}{4}
 \item The flux vector $F(M)$ of $M$ along a compact horizontal section
has nonzero horizontal component; equivalently, $F(M)=(h,0,1)$ for some $h>0$.
\item The Gaussian  curvature of $M$ is bounded and the vertical spacings
between consecutive planar ends are bounded from above and below
by positive constants, with all these constants depending only on $h$.
\item For every divergent sequence $\{ z_k\} _k\subset \C /\langle i\rangle $,
there exists a subsequence of the meromorphic functions $g_k(z)=g(z+z_k)$
which converges uniformly on compact
subsets of $\C /\langle i\rangle $ to a non-constant meromorphic function
$g_{\infty }\colon \C /\langle i\rangle \to \C \cup \{ \infty \} $
(we will refer to this property saying
that g is quasi-periodic).
In fact, $g_{\infty }$ corresponds to the
Gauss map of a minimal surface $M_{\infty }$ satisfying the same properties and normalization
(\ref{star}) as $M$, which is the limit
 of a related subsequence of translations of $M$ by vectors whose
$x_3$-components are $\mbox{\rm Re} (z_k)$.
\end{enumerate}
\end{theorem}
As said above, the proof of properties 5, 6 and 7 are more delicate than the ones in Theorem~\ref{thm5.6}
as they depend on Colding-Minicozzi theory. For instance, the fact that the Gaussian curvature
$K_M$ of $M$ is bounded with the bound depending only on an upper bound of the horizontal
component of $F(M)$, was proven in Theorem~5 of~\cite{mpr3} in the more general
case of a sequence $\{ M_k\} _k\subset {\cal P}$ as in Theorem~\ref{thm5.6}, such that
$F(M_k)=(h_k,0,1)$ and $\{ h_k\} _k$ is bounded from above. This proof
of the existence of a uniform curvature estimate is by contradiction: the existence
of a sequence $p_k\in M_k$ such that
$|K_{M_k}|(p_k)\to \infty $ creates a nonflat blow-up limit of the $M_k$
around $p_k$ with can be proven to be a vertical helicoid (this uses the uniqueness of the helicoid
among properly embedded, simply connected, nonflat minimal surfaces, see Section~\ref{sec5.2}).
A careful application of the
so called {\it Limit Lamination Theorem for Planar Domains} (Theorem~0.9 in Colding and
Minicozzi~\cite{cm25}) produces a sequence $\mu _k>0$ so that after possibly a sequence of
translations and rotations around a vertical axis, $\{ \mu _kM_k\} _k$ converges
to a foliation of $\R^3$ by horizontal planes with singular set of convergence consisting of two vertical
lines $\G \cup \G'$ separated by a positive distance.
From here one can produce a nontrivial closed curve in each $\mu _kM_k$ such that the flux vector
$F(\mu _kM_k)$ converges as $k\to \infty $ to twice the horizontal vector that joins $\G $ and $\G '$. Since
the angle between $F(M_k)$ and its horizontal projection $(h_k,0,0)$
is invariant under translations, homotheties and rotations around the $x_3$-axis, then
we contradict that $h_k$ is bounded from above.

The proof that there is a lower bound of the vertical spacings between consecutive planar ends
in item~6 of Theorem~\ref{thm5.7} follows from that fact that the boundedness of $K_M$
implies the existence of an embedded regular neighborhood of
constant radius (Meeks and Rosenberg~\cite{mr1}). The bound from above of the
same vertical spacing is again proved by contradiction, by a clever application of the
L\'opez-Ros deformation argument (see Section~\ref{sec5.1}), which also gives property~5 of Theorem~\ref{thm5.7}.
Finally, the proof of the compactness result in item~7
of Theorem~\ref{thm5.7} is essentially a consequence of the already proven
uniform bound of the Gaussian curvatures and the uniform local bounds for the area of a sequence
of translations of the surface $M$ given by item~6 of the same theorem. This completes our sketch
of proof of Theorem~\ref{thm5.7}.

As explained above, in our setting for $M\in {\cal P}$ satisfying the normalizations
in Theorems~\ref{thm5.6} and~\ref{thm5.7}, we can consider its Shiffman function $S_M$ defined
by equation (\ref{eq:Shiffman}), which is also defined on the conformal compactification of $M$
(recall that in Section~\ref{subsec:shiffman} we normalized the
height differential $\phi _3$ to be $dz$, as in Theorem~\ref{thm5.6}).
Recall also that the vanishing of the Shiffman function
is equivalent to the fact that $M$ is a Riemann minimal example. But instead of proving
directly that $S_M=0$ on $M$,
Meeks, P\'erez and Ros demonstrated that $S_M$ is a linear Jacobi function; to understand this concept, we must
first recall some basic facts about Jacobi functions on a minimal surface.

Since minimal surfaces can be viewed as critical points for the area functional $A$,
the nullity of the hessian of $A$ at a minimal surface $M$ contains
valuable information about the geometry of $M$. Normal variational fields for $M$ can be identified
with functions, and the second variation of area tells us that the functions in the
nullity of the hessian of $A$ coincide with the kernel of the {\it Jacobi operator}, which is the
Schr\"{o}dinger operator on $M$ given by
\begin{equation}
\label{eq:Jacobi}
L=\Delta - 2K_M,
\end{equation}
where $\Delta $ denotes the intrinsic Laplacian on $M$. Any function $v\in C^{\infty }(M)$
satisfying $\Delta v - 2K_Mv=0$ on $M$ is called a {\it Jacobi function,} and corresponds to an
{\it infinitesimal} deformation of $M$ by minimal surfaces. It turns out that the
Shiffman function $S_M$ is a Jacobi function, i.e., it satisfies (\ref{eq:Jacobi}) (this is
general for any minimal surface whenever $S_M$ is well-defined,
and follows by direct computation from (\ref{eq:Shiffman})).
One obvious way to produce Jacobi fields is to take the normal part of
the variational field of the variation of $M$ by moving it through a 1-parameter family of isometries
of $\R^3$. For instance, the translations $M+ta$ with $a\in \R^3$, produce the Jacobi function
$\langle N,a\rangle $ (here $N$ is the Gauss map of $M$), which is called {\it linear Jacobi function.}
One key step in the proof of Meeks, P\'erez and Ros is the following one.

\begin{proposition}
\label{prop5.8}
  Let $M\in {\cal P}$ be a surface with infinitely many ends and satisfying
the normalizations in Theorems~\ref{thm5.6} and~\ref{thm5.7}. If the Shiffman Jacobi function $S_M$
of $M$ is linear, i.e., $S_M=\langle N,a\rangle $ for some $a\in \R^3$, then $M$ is a Riemann minimal example.
\end{proposition}
The proof of Proposition~\ref{prop5.8} goes as follows. We first pass from real valued
Jacobi functions to complex valued ones by means of
the {\it conjugate of a Jacobi function.} The conjugate function of a Jacobi
function $u$ over a minimal surface $M$ is the (locally defined)
support function $u^*=\langle (X_u)^*,N\rangle $
of the conjugate surface\footnote{The conjugate surface of a minimal
surface is that one whose coordinate functions are harmonic conjugates of the coordinate
functions of the original minimal surface; the conjugate surface is only locally defined,
and it is always minimal.}  $(X_u)^*$ of the
branched minimal surface $X_u$ associated to $u$ by the so-called Montiel-Ros
correspondence~\cite{mro1}; in particular, both $X_u$ and $(X_u)^*$ have the same Gauss map $N$
as $M$, and $u^*$ also satisfies the Jacobi equation. Now suppose $M\in {\cal P}$ is an Proposition~\ref{prop5.8},
i.e., $S_M$ is linear. It is then easy to show that the conjugate Jacobi function $(S_M)^*$
of $S_M$, is also linear, from where $S_M+iS_M^*=\langle N,a\rangle $ for some $a\in \C^3$, which again by
(\ref{eq:Shiffman}), produces a complex ODE for $g$ of second order, namely
\begin{equation}
\label{eq:gOGE}
\overline{g}\left( \frac{3}{2}\frac{(g')^2}{g}-g''-B-a_3g\right) =
\frac{g''}{g}-\frac{1}{2}\left( \frac{g'}{g}\right) ^2+Ag-a_3,
\end{equation}
for some complex constants $A,B,a_3$ that only depend of $a$.
As $g$ is holomorphic and not constant, (\ref{eq:gOGE}) implies that both its
right-hand side and the expression between parenthesis in its left-hand side
vanish identically. Solving for $g''$ in both equations, one arrives to the following
complex ODE of first order:
\[
(g')^2 =g(-Ag^2 +2a_3 g+B),
 \]
which in turn says that the Weierstrass data $(g,\phi _3=dz)$ of $M$ factorizes through
the torus $\Sigma =\{ (\xi ,w)\in (\C\cup \{ \infty \} )^2\ | \ w^2=\xi(-A\xi^2+2a_3\xi+
B)\} $; in other words, we deduce that $M$ is in fact periodic under a translation,
with a quotient being a twice punctured torus. In this very particular situation, one can
apply the classification of periodic examples by Meeks, P\'erez and Ros in~\cite{mpr1} to conclude
that $M$ is a Riemann minimal example, and the proposition is proved.

In light of Proposition~\ref{prop5.8}, one way of finishing our classification problem
consists of proving the following statement, which will be proved assuming that Theorem~\ref{thm5.9}
stated immediately after it holds; the proof of Theorem~\ref{thm5.9} will be sketched later.
\begin{proposition}
\label{propult}
For every $M\in {\cal P}$ with infinitely many ends and satisfying
the normalizations in Theorems~\ref{thm5.6}, \ref{thm5.7} and in (\ref{star}),
the Shiffman Jacobi function $S_M$ of $M$ is linear.
\end{proposition}

\begin{theorem}[Theorem~5.14 in~\cite{mpr6}]
\label{thm5.9}
Given a surface $M\in {\cal P}$ with infinitely many ends and satisfying
the normalizations in Theorems~\ref{thm5.6} and~\ref{thm5.7},
there exists a 1-parameter family $\{ M_t\} _t\subset {\mathcal P}$
such that $M_0=M$ and the normal part of the variational field for this variation,
when restricted to each $M_t$, is the Shiffman function
$S_{M_{t}}$ multiplied by the unit normal vector field to $M_t$.
\end{theorem}

Before proving Proposition~\ref{propult}, some explanation about the integration of the Shiffman function $S_M$ appearing in
Theorem~\ref{thm5.9} is in order. As we explained in the paragraph just before the statement
of Proposition~\ref{prop5.8}, $S_M$ corresponds to an {\it infinitesimal} deformation of $M$
by minimal surfaces (every Jacobi function has this property). But this is very different from the
quite strong property of proving that $S_M$ can be integrated to an actual
variation $t\mapsto M_t\in {\cal P}$, as stated in Theorem~\ref{thm5.9}.
Even more, the parameter $t$ of this deformation can be proven to extend to be
a complex number in $\D (\ve )=\{ t\in \C \ | \ |t|<\ve \} $ for some $\ve >0$,
and $t\in \D (\ve )\mapsto M_t$ can be viewed as the real part of a
complex valued {\it holomorphic} curve in a certain complex variety. This is a very
special integration property for $S_M$, which we refer to by saying that
the Shiffman function can be {\it holomorphically integrated} for every surface
$M$ as in Theorem~\ref{thm5.9}.

We next give a sketch of the proof of Proposition~\ref{propult}, assuming the validity of Theorem~\ref{thm5.9}.
One fixes a flux vector $F=(h,0,1)$, $h>0$, consider the set
\[
{\cal P}_F=\{ M\in {\cal P} \mbox{ as in Theorems~\ref{thm5.6} and \ref{thm5.7} } \ | \ F(M)=F\}
\]
and maximize the spacing between the planar ends of surfaces in ${\cal P}_F$
(to do this one needs to be careful when specifying what planar ends are compared when measuring
distances; we will not enter in technical details here), which can be done by the compactness property given
in item~6 of Theorem~\ref{thm5.7}. Then one proves that any maximizer (not necessarily unique a priori)
$M_{\max }\in {\cal P}_F$ must have linear Shiffman function; the argument for this claim has two steps:
\begin{enumerate}[(S1)]
\item The assumed homomorphic integration of the Shiffman function of
$M_{\max }$ produces a  complex holomorphic
curve $t\in \D (\ve )\mapsto g_t\in {\cal W}$, with $g_0=g_{\max }$ being the Gauss map
of $M_{\max }$, where ${\mathcal W}$ is the complex manifold  of
quasi-periodic meromorphic functions on $\C /\langle i\rangle $ (in the sense
explained in item~7 of Theorem~\ref{thm5.7}) with double zeros and double poles; 
${\cal W}$ can be identified to the set of potential Weierstrass data of minimal immersions $(g,\phi _3=dz)$
with infinitely many planar ends. The fact that the period problem associated to $(g_t,dz)$ is
solved for any $t$ comes from the fact that for $t=0$, this period problem is solved (as $M_{\max }$ is a genuine
surface) and that the velocity vector of $t\mapsto g_t$ is the Shiffman
function at any value of $t$, which lies in the kernel of the period map. A similar argument shows
that not only $(g_t,dz)$ solve the period problem, but also the flux vector $F(M_t)$ is independent of
$t$, where $M_t$ is the minimal surface produced by the Weierstrass data $(g_t,dz)$ (thus $M_0=M_{\max }$).
Embeddedness of
$M_t$ is guaranteed from that of $M_{\max }$, by the application of the maximum principle for minimal
surfaces. Altogether, we deduce that $t\in \D (\ve )\mapsto M_t$ actually lies in ${\cal P}_F$, which
implies that the spacing between the planar ends of $M_t$, viewed as a function of $t$, achieves
a maximum at $t=0$.

\item As the spacing between the planar ends of $M_t$ can be viewed as a harmonic function
of $t$ (this follows from item~4 of Theorem~\ref{thm5.6}), then the maximizing property of $M_{\max }$
in the family $t\in \D (\ve )\mapsto M_t$ and the maximum principle for harmonic functions
gives that the spacing between the planar ends of $M_t$
remains constant in $t$; from here it is not difficult to deduce that $t\mapsto g_t$ is just a
translation in the cylinder $\C /\langle i \rangle $ of the zeros and poles of $g_t$, which
corresponds geometrically to the fact that $t\mapsto M_t$ is a translation in $\R^3$ of $M_{\max}$. Therefore,
the velocity vector of $t\mapsto M_t$ at $t=0$, which is the Shiffman function of $M_{\max }$, is linear.
\end{enumerate}

Once it is proven that the Shiffman function of $M_{\max }$ is linear, Proposition~\ref{prop5.8} implies that
$M_{\max }$ is a Riemann minimal example. A similar reasoning can be done for a minimizer
$M_{\min }\in {\cal P}_F$ of the spacing between planar ends, hence $M_{\min }$ is also a Riemann
minimal example. As there is only one Riemann example for each flux $F$, then we
deduce that the maximizer and minimizer are the same. In particular, every surface in ${\cal P}_F$ is both a
maximizer and minimizer and, hence, its Shiffman function is linear. This finishes the sketch of proof
of Proposition~\ref{propult}.

To finish this article, we will indicate how to demonstrate that the Shiffman function of every
surface $M\in {\cal P}$ in the hypotheses of Theorem~\ref{thm5.9} can be holomorphically integrated.
This step is where the Korteweg-de Vries equation (KdV) plays a crucial role, which we will explain now.
We recommend the interested reader to consult the excellent survey by
Gesztesy and Weikard~\cite{gewe1} for an overview of the notions and properties
that we will use in the sequel.

First we explain the connection between the Shiffman function and
the KdV equation. Let $M\in {\cal P}$ be a surface satisfying the hypotheses of
Theorem~\ref{thm5.9}, and let $S_M$ be its Shiffman function, which is globally defined. Its
conjugate Jacobi function $(S_M)^*$ is also globally defined; in fact, $(S_M)^*$ is given by minus the
real part of the expression enclosed between brackets in (\ref{eq:Shiffman}).
By the Montiel-Ros correspondence~\cite{mro1},
both $S_M$, $(S_M)^*$ can be viewed as the support functions of conjugate branched minimal immersions
$X,X^*\colon M\to \R^3$ with the same Gauss map as $M$. The
holomorphicity of $X+iX^*$ allows us to identify $S_M+iS_M^*$
with an infinitesimal deformation of the Gauss map $g$ of $M$ in the space
${\mathcal W}$ of quasi-periodic meromorphic functions on $\C /\langle i\rangle $ that appears in
step (S1) above. In other words,
$S_M+iS_M^*$ can be viewed as the derivative $\dot{g}_S=\left. \frac{d}{dt}\right| _{t=0}g_t$
of a holomorphic
curve $t\in \D (\ve )=\{ t\in \C \ | \ |t|< \ve \} \mapsto g_t\in
{\mathcal W}$ with $g_0=g$, which can be explicitly computed from
(\ref{eq:Shiffman}) as
\begin{equation}
\label{gpuntodeShiffman}
\dot{g}_S=\frac{i}{2}\left( g'''-3\frac{g'g''}{g}+\frac{3}{2}\frac{(g')^3}{g^2}\right) .
\end{equation}
Therefore, to integrate $S_M$ holomorphically one needs to find a holomorphic curve $t\in \D (\ve )
\mapsto g_t\in {\mathcal W}$
with $g_0=g$, such that for all $t$, the pair $(g_t,\phi _3=dz)$ is the Weierstrass data  of a minimal
surface $M_t\in {\mathcal P}$ satisfying the conditions of Theorem~\ref{thm5.9}, and
such that for every value of $t$,
\[
\left. \frac{d}{dt}\right| _{t}g_t=\frac{i}{2}\left( g_t'''-3\frac{g_t'g_t''}{g_t}+
\frac{3}{2}\frac{(g_t')^3}{g_t^2}\right) .
\]

Viewing (\ref{gpuntodeShiffman}) as an evolution equation in complex
time $t$, one could apply general techniques to find solutions $g_t=g_t(z)$
defined locally around a point $z_0\in (\C /\langle i \rangle )-g^{-1}(\{ 0,\infty \} )$
with the initial condition $g_0=g$, but such solutions are not necessarily defined
on the whole cylinder, can develop essential singularities,
and even if they were meromorphic on $\C /\langle i\rangle $, it is not clear
{\it a priori} that they would have only double zeros and poles and other properties necessary
to give rise, via the Weierstrass representation with height differential $\phi _3=dz$,
to minimal surfaces $M_t$ in ${\mathcal P}$ with infinitely many planar ends.
Fortunately, all of these problems can be solved by arguments related to the
theory of the meromorphic KdV equation.

The change of variables
\begin{equation}
\label{u}
u =-\frac{3(g')^2}{4g^2}+\frac{g''}{2g}.
\end{equation}
transforms (\ref{gpuntodeShiffman}) into the evolution equation
\begin{equation}
\label{kdv}
\frac{\partial u}{\partial t} = -u'''-6 u u',
\end{equation}
which is the celebrated KdV equation\footnote{One can find different
normalizations of the KdV equation in the literature,
given by different coefficients for $u''', uu'$ in equation
(\ref{kdv}); all of them are equivalent up to a change of variables.}.
The apparently magical change of variables (\ref{u}) has a natural explanation:
the change of variables $x=g'/g$
transforms the expression (\ref{gpuntodeShiffman}) for $\dot{g}_S$
into the evolution equation
\[
\dot{x}=\frac{i}{2}(x'''-\frac{3}{2}x^2x'),
\]
which is called a {\it modified KdV equation} (mKdV).
It is well-known that mKdV equations in $x$ can
be transformed into KdV equations in $u$ through the so called
{\it Miura transformations,} $x\mapsto u=ax'+bx^2$ with $a,b$ suitable
constants, see for example~\cite{gewe1} page~273.
Equation~(\ref{u}) is nothing but the composition of $g\mapsto x$
and a Miura transformation. The holomorphic integration of the
Shiffman function $S_M$ could be performed just in terms of the theory
of the mKdV equation, but we will instead use the more standard
KdV theory.

Coming back to the holomorphic integration of $S_M$, this
problem  amounts to solving globally in $\C /\langle i\rangle $
the Cauchy problem  for equation (\ref{kdv}), i.e.,
\begin{quote}
{\sc Problem.} Find a meromorphic solution $u(z,t)$ of (\ref{kdv}) defined for
$z\in \C /\langle i\rangle $ and $t\in \D (\ve )$, whose initial condition is
$u(z,0)=u(z)$ given by~(\ref{u}).
\end{quote}

It is a well-known fact in KdV
theory (see for instance~\cite{gewe1} and also see Segal and
Wilson~\cite{SeWi}) that the above Cauchy problem can be solved
globally producing a holomorphic curve $t\mapsto u_t$ of meromorphic
functions $u(z,t)=u_t(z)$ on $\C /\langle i\rangle $ with
controlled Laurent expansions in poles of $u_t$, provided that the
initial condition $u(z)$ is an {\it algebro-geometric potential} for
the KdV equation (to be defined below); a different question is whether or not this family $u_t(z)$
solves our geometric problem related to minimal surfaces in ${\cal P}$.

To understand the notion of algebro-geometric potential, one must view
(\ref{kdv}) as the level $n=1$ in a sequence of evolution equations
in $u$, called the {\it KdV hierarchy,}
\begin{equation}
\label{kdvn}
\left\{ \frac{\partial u}{\partial t_n} = -\partial_z{\mathcal P}_{n+1}(u)\right\} _{n\geq 0},
\end{equation}
where ${\mathcal P}_{n+1}(u)$ is a differential operator given by a polynomial expression of $u$
and its derivatives with respect to $z$ up to order $2n$. These operators,
which are closely related
to Lax Pairs (see Section~2.3 in~\cite{gewe1}) are defined by
the recurrence law
\begin{eqnarray}
\label{law}
\left\{ \begin{array}{l}
\partial_z {\mathcal P}_{n+1}(u) = (\partial_{zzz} + 4u\,\partial_z+2u'){\mathcal P}_{n}(u), \\
\rule{0cm}{.5cm}{\mathcal P}_{0}(u)=\frac{1}{2}.
\end{array}\right.
\end{eqnarray}
In particular, ${\mathcal P}_1(u)=u$ and ${\mathcal P}_2(u)=u''+3u^2$
(plugging ${\mathcal P}_2(u)$ in (\ref{kdvn}) one obtains the KdV equation).
Hence, for each $n\in \N \cup \{ 0\} $, one must consider the
right-hand side of the $n$-th equation
in (\ref{kdvn}) as a polynomial expression of $u=u(z)$
and its derivatives with respect to $z$ up to order $2n+1$. We will
call this expression a {\it flow}, denoted by $\frac{\partial u}{\partial t_n}$.
A function $u(z)$ is said to be an
{\it algebro-geometric potential} of the KdV equation if there exists a flow
$\frac{\partial u}{\partial t_n}$  which is a linear combination of the
lower order flows in the KdV  hierarchy.

Once we understand the notion of algebro-geometric potential for the KdV equation,
we have divided our goal of proving the holomorphic integration of the Shiffman
function for any surface $M$ as in Theorem~\ref{thm5.9} into two final steps.

\begin{enumerate}[(T1)]
\item For every minimal surface $M\in {\cal P}$ satisfying the hypotheses of Theorem~\ref{thm5.9},
the function $u=u(z)$ defined by equation (\ref{u}) in terms of the Gauss map $g(z)$
of $M$, is an algebro-geometric potential of the KdV equation. This step would then give
a meromorphic solution $u(z,t)=u_t(z)$ of the KdV flow (\ref{kdv}) defined for
$z\in \C /\langle i\rangle $ and $t\in \D (\ve )$, with initial condition
$u(z,0)=u(z)$ given by~(\ref{u}).

\item With $u_t(z)$ as in (T1), it is possible to define a holomorphic curve $t\mapsto g_t\in {\cal W}$
with $g_0=g$ (recall that $g$ is the stereographic projection of the Gauss map of $M$), such that
$(g_t,\phi _3=dz)$ solves the period problem and defines a minimal surface $M_t\in {\cal P}$
that satisfies the conclusions of Theorem~\ref{thm5.9}.
\end{enumerate}

Property (T1) follows from a combination of the following two facts:
\begin{enumerate}[(a)]
\item Each flow $\frac{\partial u}{\partial t_n}$ in the KdV hierarchy
(\ref{kdvn}) produces a {\it bounded,} complex valued Jacobi function $v_n$ on
$\C /\langle i\rangle $ in a similar manner to the way that the flow
$\frac{\partial u}{\partial t_1}$ produces the complex Shiffman function $S_M+iS_M^*$.
\item Since the Jacobi functions $v_n$ produced in item~(a) are bounded on
$\C /\langle i\rangle $ and the Jacobi operator (\ref{eq:Jacobi}) is the Shr\"{o}dinger
operator given by (\ref{eq:Jacobi}) on $M$, then the $v_n$ can be considered to lie
in  the kernel of a Schr\"{o}dinger operator $L_M$ on $\C /\langle
i\rangle $ with bounded potential; namely, $L_M=(\Delta_{\esf
^1}+\partial^2_t)+V_M$ where $\C /\langle i\rangle $ has been
isometrically identified with $\esf^1\times \R$ endowed with the usual
product metric $d\theta ^2\times dt^2$, and the potential $V_M$ is
the square of the norm of the differential of the Gauss map of $M$
with respect to $d\theta ^2\times dt^2$ ($V_M$ is bounded since $M$
has bounded Gaussian curvature by item~6 of Theorem~\ref{thm5.7}).
Finally, the kernel of $L_M$
restricted to bounded functions
 is finite dimensional; this finite dimensionality was proved
by Meeks, P\'erez and Ros\footnote{Following arguments
by Pacard (personal communication), which in turn are inspired
in a paper by Lockhart and McOwen~\cite{loMcOw1}.}
in~\cite{mpr6} and also follows from a more general
result by Colding, de Lellis and Minicozzi~\cite{cm39}).
\end{enumerate}

As for the proof of property (T2) above,
the aforementioned control on the Laurent expansions in poles of $u_t$
coming from the integration of the Cauchy problem for the KdV
equation, is enough to prove that the corresponding
meromorphic function $g_t$ associated to $u_t$ by equation (\ref{u}) has the
correct behavior in poles and zeros; this
property together with the fact that both $S_M,S_M^*$ preserve
infinitesimally the complex periods along any closed
curve in $\C /\langle i \rangle $, suffice to show that the Weierstrass
data $(g_t,\phi _3=dz)$ solves the period problem for every $t$ and has the same flux vector $F=(h,0,1)$
as the original $M$, thereby giving rise to a surface $M_t\in {\mathcal P}$
with the desired properties. This finishes our sketch of proof of the holomorphic integration of
the Shiffman function of an arbitrary surface $M\in {\cal P}$ satisfying the hypotheses
of Theorem~\ref{thm5.9}.

\begin{remark} {\em
While the classification problem for properly embedded minimal planar domains stated
at the beginning of Section 5 has been completed, a natural
and important generalization to it remains open:

\begin{quote}
{\it
Problem: Classify all complete embedded minimal planar domains in $\R^3$.
}
\end{quote}

This more general classification question would be resolved if we knew a priori that
any complete embedded minimal surface $M$ of finite genus in $\rth$ is  proper. The
conjecture that this properness property holds
for such an $M$ is called the Embedded Calabi-Yau
Conjecture for Finite Genus.  In their ground breaking work in~\cite{cm35}, Colding and Minicozzi
solved this conjecture in the special case that the minimal surface $M$ has finite topology.
More recently, Meeks, P\'erez and Ros~\cite{mpr9} proved that
the conjecture holds if and only it $M$ has a countable number of ends. However, as
of the writing of this manuscript,
the solution of the Embedded Calabi-Yau
Conjecture for Finite Genus remains unsettled.

}
\end{remark}

\bibliographystyle{plain}
\bibliography{bill}

\begin{thebibliography}{10}

\bibitem{chm3}
M.~Callahan, D.~Hoffman, and W.~H. Meeks~III.
\newblock The structure of singly-periodic minimal surfaces.
\newblock {\em Invent. Math.}, 99:455--481, 1990.
\newblock MR1032877, Zbl 695.53005.

\bibitem{cm39}
T.~H. Colding, C.~de~Lellis, and W.~P. Minicozzi~II.
\newblock Three circles theorems for {S}chr\"{o}dinger operators on cylindrical
  ends and geometric applications.
\newblock {\em Comm. Pure Appl. Math.}, 61(11):1540--1602, 2008.
\newblock MR2444375, Zbl pre05358518.

\bibitem{cm23}
T.~H. Colding and W.~P. Minicozzi~II.
\newblock The space of embedded minimal surfaces of fixed genus in a
  $3$-manifold {I}{V}; {L}ocally simply-connected.
\newblock {\em Ann. of Math.}, 160:573--615, 2004.
\newblock MR2123933, Zbl 1076.53069.

\bibitem{cm35}
T.~H. Colding and W.~P. Minicozzi~II.
\newblock The {C}alabi-{Y}au conjectures for embedded surfaces.
\newblock {\em Ann. of Math.}, 167:211--243, 2008.
\newblock MR2373154, Zbl 1142.53012.

\bibitem{cm25}
T.~H. Colding and W.~P. Minicozzi~II.
\newblock The space of embedded minimal surfaces of fixed genus in a
  $3$-manifold {V}; {F}ixed genus.
\newblock {\em Ann. of Math.}, 181(1):1--153, 2015.
\newblock MR3272923, Zbl 06383661.

\bibitem{col1}
P.~Collin.
\newblock Topologie et courbure des surfaces minimales de $\rth$.
\newblock {\em Ann. of Math. (2)}, 145--1:1--31, 1997.
\newblock MR1432035, Zbl 886.53008.

\bibitem{ckmr1}
P.~Collin, R.~Kusner, W.~H. Meeks~III, and H.~Rosenberg.
\newblock The geometry, conformal structure and topology of minimal surfaces
  with infinite topology.
\newblock {\em J. Differential Geom.}, 67:377--393, 2004.
\newblock MR2153082, Zbl 1098.53006.

\bibitem{dhkw1}
U.~Dierkes, S.~Hildebrandt, A.~K\"{u}ster, and O.~Wohlrab.
\newblock {\em Minimal Surfaces I}.
\newblock Grundlehren der mathematischen {W}issenschaften 295. Springer-Verlag,
  1992.
\newblock MR1215267, Zbl 0777.53012.

\bibitem{en1}
A.~Enneper.
\newblock Die cyklischen {F}l\"{a}chen.
\newblock {\em Z. Math. und Phys.}, 14:393--421, 1869.
\newblock JFM 02.0585.01.

\bibitem{fme2}
C.~Frohman and W.~H.~Meeks III.
\newblock The ordering theorem for the ends of properly embedded minimal
  surfaces.
\newblock {\em Topology}, 36(3):605--617, 1997.
\newblock MR1422427, Zbl 878.53008.

\bibitem{gewe1}
F.~Gesztesy and R.~Weikard.
\newblock Elliptic algebro-geometric solutions of the {K}d{V} and {AKNS}
  hierarchies---an analytic approach.
\newblock {\em Bull. Amer. Math. Soc. (N.S.)}, 35(4):271--317, 1998.
\newblock MR1638298, Zbl 0909.34073.

\bibitem{hm10}
D.~Hoffman and W.~H. Meeks~III.
\newblock The strong halfspace theorem for minimal surfaces.
\newblock {\em Invent. Math.}, 101:373--377, 1990.
\newblock MR1062966, Zbl 722.53054.

\bibitem{hu1}
A.~Huber.
\newblock On subharmonic functions and differential geometry in the large.
\newblock {\em Comment. Math. Helvetici}, 32:181--206, 1957.
\newblock MR0094452, Zbl 0080.15001.

\bibitem{ja2}
W.~Jagy.
\newblock Minimal hypersurfaces foliated by spheres.
\newblock {\em Michigan Math. J.}, 38(2):255--270, 1991.
\newblock MR1098859, Zbl 0725.53061.

\bibitem{Kh1}
S.~Khan.
\newblock A minimal lamination of the unit ball with singularities along a line
  segment.
\newblock {\em Illinois J. Math.}, 53(3):833--855, 2009.
\newblock MR2727357, Zbl 1225.53009.

\bibitem{kl1}
S.~Kleene.
\newblock A minimal lamination with {C}antor set-like singularities.
\newblock {\em Proc. Am. Math. Soc.}, 140(4):1423--1436, 2012.
\newblock MR2869127, Zbl pre06028355.

\bibitem{loMcOw1}
R.~B. Lockhart and R.~C. Mc{O}wen.
\newblock Elliptic differential operators on noncompact manifolds.
\newblock {\em Ann. Scuola Norm. Sup. Pisa}, 12(3):409--447, 1985.
\newblock MR0837256, Zbl 0615.58048.

\bibitem{lor1}
F.~J. L\'opez and A.~Ros.
\newblock On embedded complete minimal surfaces of genus zero.
\newblock {\em J. Differential Geom.}, 33(1):293--300, 1991.
\newblock MR1085145, Zbl 719.53004.

\bibitem{mape1}
F.~Martin and J.~P\'erez.
\newblock Superficies minimales foliadas por circunferencias: los ejemplos de
  {R}iemann.
\newblock {\em Gaceta de la RSME}, 6(3):1--27, 2003.

\bibitem{me25}
W.~H. Meeks~III.
\newblock The regularity of the singular set in the {C}olding and {M}inicozzi
  lamination theorem.
\newblock {\em Duke Math. J.}, 123(2):329--334, 2004.
\newblock MR2066941, Zbl 1086.53005.

\bibitem{me30}
W.~H. Meeks~III.
\newblock The limit lamination metric for the {C}olding-{M}inicozzi minimal
  lamination.
\newblock {\em Illinois J. of Math.}, 49(2):645--658, 2005.
\newblock MR2164355, Zbl 1087.53058.

\bibitem{mpe1}
W.~H. Meeks~III and J.~P\'{e}rez.
\newblock Conformal properties in classical minimal surface theory.
\newblock In {\em Surveys of Differential Geometry IX - Eigenvalues of
  Laplacian and other geometric operators}, pages 275--336. International
  Press, edited by Alexander Grigor'yan and Shing Tung Yau, 2004.
\newblock MR2195411, Zbl 1086.53007.

\bibitem{mpr9}
W.~H. Meeks~III, J.~P\'{e}rez, and A.~Ros.
\newblock The embedded {C}alabi-{Y}au conjectures for finite genus.
\newblock Work in progress.

\bibitem{mpr10}
W.~H. Meeks~III, J.~P\'{e}rez, and A.~Ros.
\newblock Local removable singularity theorems for minimal laminations.
\newblock To appear in J. Differential Geom. Preprint available at
  http://wdb.ugr.es/$\sim $jperez/publications-by-joaquin-perez/.

\bibitem{mpr11}
W.~H. Meeks~III, J.~P\'{e}rez, and A.~Ros.
\newblock Structure theorems for singular minimal laminations.
\newblock Preprint available at
  http://wdb.ugr.es/local/jperez/publications-by-joaquin-perez/.

\bibitem{mpr1}
W.~H. Meeks~III, J.~P\'{e}rez, and A.~Ros.
\newblock Uniqueness of the {R}iemann minimal examples.
\newblock {\em Invent. Math.}, 133:107--132, 1998.
\newblock MR1626477, Zbl 916.53004.

\bibitem{mpr3}
W.~H. Meeks~III, J.~P\'{e}rez, and A.~Ros.
\newblock The geometry of minimal surfaces of finite genus {I}; curvature
  estimates and quasiperiodicity.
\newblock {\em J. Differential Geom.}, 66:1--45, 2004.
\newblock MR2128712, Zbl 1068.53012.

\bibitem{mpr4}
W.~H. Meeks~III, J.~P\'{e}rez, and A.~Ros.
\newblock The geometry of minimal surfaces of finite genus {I}{I}; nonexistence
  of one limit end examples.
\newblock {\em Invent. Math.}, 158:323--341, 2004.
\newblock MR2096796, Zbl 1070.53003.

\bibitem{mpr18}
W.~H. Meeks~III, J.~P\'{e}rez, and A.~Ros.
\newblock Limit leaves of an {H} lamination are stable.
\newblock {\em J. Differential Geom.}, 84(1):179--189, 2010.
\newblock MR2629513, Zbl 1197.53037.

\bibitem{mpr6}
W.~H. Meeks~III, J.~P\'{e}rez, and A.~Ros.
\newblock Properly embedded minimal planar domains.
\newblock {\em Ann. of Math.}, 181(2):473--546, 2015.
\newblock MR3275845, Zbl 06399442.

\bibitem{mr1}
W.~H. Meeks~III and H.~Rosenberg.
\newblock The maximum principle at infinity for minimal surfaces in flat
  three-manifolds.
\newblock {\em Comment. Math. Helvetici}, 65:255--270, 1990.
\newblock MR1057243, Zbl 713.53008.

\bibitem{mr8}
W.~H. Meeks~III and H.~Rosenberg.
\newblock The uniqueness of the helicoid.
\newblock {\em Ann. of Math.}, 161:723--754, 2005.
\newblock MR2153399, Zbl 1102.53005.

\bibitem{mr13}
W.~H. Meeks~III and H.~Rosenberg.
\newblock The minimal lamination closure theorem.
\newblock {\em Duke Math. Journal}, 133(3):467--497, 2006.
\newblock MR2228460, Zbl 1098.53007.

\bibitem{mro1}
S.~Montiel and A.~Ros.
\newblock Schr\"{o}dinger operators associated to a holomorphic map.
\newblock In {\em Global Differential Geometry and Global Analysis (Berlin,
  1990)}, volume 1481 of {\em Lecture Notes in Mathematics}, pages 147--174.
  Springer-Verlag, 1991.
\newblock MR1178529, Zbl 744.58007.

\bibitem{ni2}
J.~{C.}~{C.} Nitsche.
\newblock {\em Lectures on Minimal Surfaces}, volume~1.
\newblock Cambridge University Press, 1989.
\newblock MR1015936, Zbl 0688.53001.

\bibitem{os1}
R.~Osserman.
\newblock {\em A Survey of Minimal Surfaces}.
\newblock Dover Publications, New York, 2nd edition, 1986.
\newblock MR0852409, Zbl 0209.52901.

\bibitem{ri2}
B.~Riemann.
\newblock \"{U}ber die {F}l\"{a}che vom kleinsten {I}nhalt bei gegebener
  {B}egrenzung.
\newblock {\em Abh. K\"{o}nigl, d. Wiss. G\"{o}ttingen, Mathem. Cl.}, 13:3--52,
  1867.
\newblock K. Hattendorf, editor. JFM 01.0218.01.

\bibitem{SeWi}
G.~Segal and G.~Wilson.
\newblock Loop groups and equations of ${K}d{V}$ type.
\newblock {\em Pub. Math. de I.H.E.S.}, 61:5--65, 1985.
\newblock MR0783348, Zbl 0592.35112.

\bibitem{sh1}
M.~Shiffman.
\newblock On surfaces of stationary area bounded by two circles, or convex
  curves, in parallel planes.
\newblock {\em Ann. of Math.}, 63:77--90, 1956.
\newblock MR0074695, Zbl 0070.16803.

\end{thebibliography}

\end{document}